\documentclass[makeidx]{amsart}
\usepackage[active]{srcltx}
\usepackage{ragged2e}
\usepackage{ulem}
\usepackage{color}
\usepackage{epsfig}
\usepackage{hyperref}
\usepackage{amsmath}
\usepackage{amssymb}
\usepackage{hyperref}
\newtheorem{Proposition}{Proposition}
  \newtheorem{Remark}{Remark}
 \newtheorem{Assumption}{Assumption}
  \newtheorem{Corollary}[Proposition]{Corollary}
  \newtheorem{Lemma}[Proposition]{Lemma}
  
  \newtheorem{Theorem}{Theorem}

 \newtheorem{Note}[Remark]{Note}

\newcommand {\z}{{\noindent}}

\def\b#1{\textcolor{black}{#1}}

\def\epsilon{\varepsilon}
\def\ve{\varepsilon}
 \def\HH{\mathbb{H}}
\def\CC{\mathbb{C}}
 \def\RR{\mathbb{R}}

\def\Re{\mathrm{Re}}

\def\ds{\displaystyle}
 \def\({\left(} \def\){\right)} \makeindex
\author{O. Costin, M. Huang} \address{Mathematics Department\\The Ohio State University\\Columbus, OH 43210}\address{Department of Mathematics\\City University of Hong Kong\\HONG KONG
}
\title{Decay estimates for One-dimensional wave equations with inverse power potentials}

\begin{document}
$ $ \vskip -0.2cm
\begin{abstract}
  We study the one-dimensional wave equation with an inverse power potential
  that equals $const.x^{-m}$ for large $|x|$ where
  $m$ is any positive integer greater than or equal to 3. We show that the solution decays pointwise like
  $t^{-m}$ for large $t$, which is consistent with existing
  mathematical and physical literature under slightly different assumptions
  (see e.g. \cite{phy1}, \cite{ds1}, \cite{sch1}).

  Our results can be generalized to potentials consisting of a finite
  sum of inverse powers, the largest of which being
  $const.x^{-\alpha}$ where $\alpha>2$ is a real number, as well as potentials of the form $const.x^{-m}+O( x^{-m-\delta_1})$ with $\delta_1>3$.
\end{abstract}
\vskip -2cm
\maketitle
\section{Introduction}
There is an extensive literature - both mathematical and physical- on the
on decay estimates for wave equations and Schr\"odinger equations with a potential, starting with Strichartz's work for $V=0$ \cite{Strichartz,Strichartz2}. The case $V\ne 0$, important in physics,  is the subject of many recent papers where a variety of modern analytical tools and different assumptions on $V$  are used (see \cite{Andersson,ds1,dancona1,dancona2,price} and also
 \cite{sch1} for a survey).

In the physical community the corresponding
problem goes by the name of tails, and the precise description of these tails is an important issue in scattering
theory. Based on nonrigorous and numerical methods, physicists predicted that the solutions to wave equations on the line with potentials decaying like
$|x|^{-\alpha}$ as $|x|\to\infty$ will decay in time like $t^{-\alpha}$, see for example \cite{phy1,phy2}.
Mathematically, a recent study by R. Donninger and W. Schlag (\cite{ds1})
showed that for potentials $V(x)$ decaying like $|x|^{-\alpha}$
where $2<\alpha\leqslant 4$ with no bound state or zero energy resonance, the solution $\psi$ to the one-dimensional wave equation \begin{equation}
\label{wave}
\frac{\partial^2\psi(x,t)}{\partial t^2}-\frac{\partial^2\psi(x,t)}{\partial x^2}+V(x)\psi(x,t)=0
\end{equation}is bounded by
$t^{-\alpha}$ for large $t$. They also obtained a similar estimate for the more important Regge-Wheeler potential, though it is not known whether the estimate is sharp (see \cite{price}).

The purpose of this paper is to  give \textcolor{black}{\it sharp estimates} for the decay of $\psi$
 where $V(x)=const.x^{-m}$ for large $|x|$ (the constants are allowed to be different for positive and
negative $x$) where $m\in \mathbb{N}$ and $m\geqslant3$. The result is consistent
with \cite{ds1} and confirms the predictions by physicists.

Our method is based on inverse Laplace transform of the equation in $t$, a technique first used
to study the time decay of Schr\"{o}dinger equations (see e.g. \cite{om1,CMP1}), and it can applied to potentials consisting of a finite
  sum of inverse powers, the largest of which being
  $const.x^{-\alpha}$ where $\alpha>2$ is a real number, as well as potentials of the form $const.x^{-m}+O( x^{-m-\delta_1})$ with $\delta_1>3$ (see Section \ref{extend}). The advantage of our approach is that it gives sharp estimates based on explicit calculations. A further refinement of this approach is expected to lead to a proof of Price's Law on Schwarzschild black holes (see e.g. \cite{price}).

\section{Setting and main results}
 We analyze the wave equation \eqref{wave} under the assumptions:
\begin{Assumption} \label{Assu1}{\rm
(i) The potential $V$ is such that the one-dimensional Schor\"odinger
  operator $A:=-\frac{d^2}{dx^2}+V(x)$ has no bound states
and no zero energy resonances.

(ii) $V$ is $ m+2$ times differentiable.

(iii) As $x\to\pm\infty$ we have $V(x)=const._{\pm}x^{-m}$  where $m\in \mathbb{N}$ and $m\geqslant3$.}
\end{Assumption}
\z The solution to \eqref{wave} (cf. \cite{ds1}) can be written as
$$\psi(t)=\cos(t\sqrt{A})\psi_0+\frac{\sin(t\sqrt{A})}{\sqrt{A}},\psi_1\ \ \  \psi_0(x):=\psi(x,0),\ \psi_1(x)=\ds\frac{\partial\psi(x,0)}{\partial t}
$$ where $\psi_{0,1}\in L^2(\mathbb{R})$.

 Our main results are

\begin{Theorem}\label{T1} Under Assumption \ref{Assu1} we have
\label{xbounds}
$$\frac{\sin(t\sqrt{A})}{\sqrt{A}}\psi_1(x)=\hat{r}_1(x) \langle t \rangle^{-m}+\langle t \rangle^{-m}R_1(x,t)$$
$$\cos(t\sqrt{A})\psi_0(x)=\hat{r}_0(x)\langle t \rangle^{-m-1}+\langle t \rangle^{-m-1}R_0(x,t)$$
 where
$$||\langle x\rangle^{-2}\hat{r}_{j}(x)||_{\infty}\lesssim ||\langle
x\rangle^2\psi_{j}(x)||_{1},\ j=0,1$$
$$||\langle x\rangle^{-m-2}R_1(x,t)||_{\infty}\leqslant ||\langle x\rangle^{m+2}\psi_1(x)||_{1}$$
$$||\langle x\rangle^{-m-3}R_0(x,t)||_{\infty}\leqslant  ||\langle x\rangle^{m+3}\psi_0(x)||_{1}+||\langle x\rangle^{m+3}\psi_0'(x)||_{1}$$
Here $\langle x\rangle:=(1+x^2)^{1/2}$, and for $j=1,2$ the infinity norms of $R_{j}(x,t),\,j=1,2$ are in both $x$ and $t$, and $\lim_{t\to\infty}R_{j}(x,t)=0,$. Moreover, $\hat{r}_j(x)$ are nonzero for generic initial data (cf. Remark \ref{gende}).
\end{Theorem}

In Section \ref{extend} we discuss generalizations where $V$ is a sum of inverse powers, and an extension of results of the type in \cite{sch1}. The special case $m=2$ will also be briefly discussed in Note \ref{m2}
below.

The basic strategy we use  is to take the Laplace transform in $t$ of \eqref{wave} and study the solutions of the transformed equation.
Laplace transformability is shown in  Proposition \ref{lapt} in the Appendix;
its existence does not require Assumption \ref{Assu1} (i);  the result of
Theorem \ref{T1} is however contingent on it.

\section{Discussion of Methods and Main Steps of Proof}
We use integral transforms to regularize the problem. First we take the Laplace transform \b{($\mathcal{L}$)} in $t$ of equation \eqref{wave}, which transforms \eqref{wave} into an ODE (see \eqref{psiwave} below) in the dual variable $\epsilon$. \b{The position of the singularities of the solution of  \eqref{psiwave}  indicates possible exponential behavior and oscillations  while the type of singularity is related to the type of power law decay.} \b{As a very simple illustration where  the duality between decay in $t$ and singularities in $\epsilon$  is manifest, consider the function $f(t)=t^{-\beta}e^{-\alpha t} $ with $0\le \beta\le 1$ and where $\alpha$ can have nonzero imaginary part. The Laplace transform of $f$ is
$$ \mathcal{L}(f)=\hat{f}=\Gamma (1-\beta )(\varepsilon +\alpha)^{\beta-1}$$}
 \b{The asymptotics in $t$ follows from Watson's lemma, after deforming the contour as in Fig. 1, where we took for definiteness $\epsilon=0$. In the actual problem, the only singularity contributing to the asymptotics is indeed $\epsilon=0$, which is a branch point. The type of the singularity follows from the small $\epsilon$ asymptotic behavior of the associated homogenous equation \eqref{eq:hom1}.  The asymptotic  analysis {\it in $\epsilon$} of the full problem is complicated by the fact that, when $\epsilon$ tends to 0 and $x$ goes to $\infty$, there is a competition between $V(x)$ and $\epsilon^2$ and as a result, this is effectively a singularly perturbed problem.  To regularize it, we apply another transform to \eqref{eq:hom1}, an inverse Laplace transform in $x$. The resulting equation is the ODE \eqref{eq:eqdiff} with only regular singular points. Global analysis of \eqref{eq:eqdiff} involves long but conceptually relatively simple calculations, which are carried out in Section \ref{propers}. The singularity structure of the Laplace transform of the solution $\psi$ of \eqref{psiwave} is obtained in Section \ref{prothm},  where the main results are the expansions \eqref{hardeq} and \eqref{hardeq0}, which are valid for all $\epsilon\in \overline{\mathbb{H}}\backslash \{0\}$ and $x\in\mathbb{R}$. These expansions contain the crucial leading terms $r_3(x)\frac{\epsilon^{m-1}\ln\epsilon}{(1+\epsilon\langle x\rangle)^{m+2}}$ and $r_4(x)\frac{\epsilon^{m}\ln\epsilon}{(1+\epsilon\langle x\rangle)^{m+3}}$ that yield the desired time decay, which is shown in Section \eqref{timedec}. The terms in the inverse Laplace transform formula \eqref{finalpre} are estimated differently: the contribution of the leading term is found by contour deformation--see Fig 1, while the remainder $R(x,\epsilon)$ is simply estimated by integration by parts, while the last term is estimated using \eqref{g0est}}.
\begin{figure}
\input{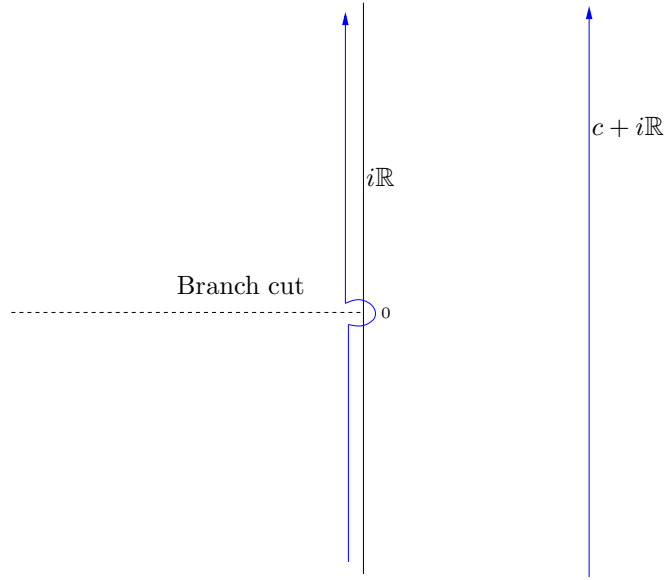}
\label{Fig1}
\caption{\textcolor{black}{Deformation of the inverse Laplace contour for the main singular terms; the leftmost contour allows for obtaining sharp estimates using Watson's lemma. (The point $-\frac12$ is not special--it is chosen for convenience.)}}
\end{figure}
\section{Regularity Properties of the Laplace Transformed Equation}\label{propers}

By taking the $t$-Laplace transform of  \eqref{wave} we obtain the ODE
\begin{equation}
\label{psiwave}
\hat{\psi}''(x,\epsilon)=\left(V(x)+\epsilon^2\right)\hat{\psi}(x,\epsilon)+\psi_1(x)+\epsilon\psi_0(x)
\end{equation}
where $$\hat{\psi}(x,\epsilon)=\int_0^{\infty}e^{-\epsilon t}\psi(x,t)dt$$

The analysis relies on properties of the exponentially decaying solutions of the homogeneous equation
\begin{equation}
 \label{eq:hom1}
  y''(x)=\left(V(x)+\epsilon^2\right)y(x)
\end{equation}
For $x>0$  and $\epsilon>0$ the exponentially decaying solution has the behavior
  \begin{equation}
    \label{eq:defs}
y(x)=y_+(x)= e^{-\epsilon x}(1+s(x;\epsilon))
  \end{equation}
where $s(x;\epsilon)$ is
  an $o(1)$ power series in $1/x$, as $x\to\infty$.

Similarly, the the exponentially decaying solution for $x<0$  and $\epsilon>0$ satisfies
$$y_-(x)\textcolor{black}{=} e^{\epsilon x}(1+s_-(x;\epsilon))=e^{-\epsilon x}(1+o(1))$$
The solution of the Laplace transformed equation \eqref{psiwave} can be written as
\begin{multline}
\label{invlap}
\hat{\psi}(x,\epsilon)=\mathcal{G}(\psi_1+\epsilon\psi_0);\
\mathcal{G}(\psi):=\frac{1}{W(\epsilon)}\(y_-(x)\int_{\infty}^{x}y_+(u)\psi(u)du-y_+(x)\int_{-\infty}^{x}y_-(u)\psi(u)du\)
\end{multline}
where $W=y_+y'_--y'_+y_-$ is the Wronskian.

\begin{Note}\label{m2}
For $V(x)=a/x^m$ and $m=1, 2$ the equation
can be solved in terms of special functions.

For $V(x)=a/x^2$, the
solution that decays like $e^{-\epsilon x}$ as $x\to \infty$
is given in terms of the modified Bessel function $K$ as
$$y^+=\sqrt{2\epsilon x/\pi}K_{\alpha}(\epsilon x);\ \alpha =\sqrt{a+1/4}$$
For small $\epsilon$ and fixed $x$, $y^+$ has the form
\begin{equation}
  \label{loc}
  C_1(x) \epsilon^{1+\alpha}A_1(\epsilon)+
 C_2(x) \epsilon^{1-\alpha}A_2(\epsilon)
\end{equation}
with $A_1,A_2$ analytic.

\end{Note}

For $m\geqslant3$, the exponentially decaying functions have asymptotic expansions in integer powers of $\epsilon$ and $\epsilon\ln\epsilon$.
By symmetry it is sufficient to study the case $x>0$.
The function $s$ satisfies
\begin{equation}
  \label{eq:eqs}
  s''-2\epsilon s'-V(x)s=V(x)
\end{equation}
It is easy to see
  that an $o(1)$ solution of (\ref{eq:eqs}) exists and is unique.

The remaining of this section is dedicated to detailed calculations to obtain the regularity properties of $y(x;\epsilon)$ in \eqref{eq:defs}, in particular near the branch point $\epsilon=0$. In Section \ref{pros} we take the inverse Laplace transform of \eqref{eq:eqs}, study the transformed function $\hat{s}$, and obtain the expansions in Corollary \ref{coro1}. In Section \ref{smallepss} we calculate the small $\epsilon$ behavior of $s(x;\epsilon)$ (expansion \eqref{newad1} in Lemma \ref{lem5}). In Section \ref{ypm} we obtain the small $\epsilon$ behavior of $y(x;\epsilon)$. Finally in Section \ref{yleps} we obtain the regularity properties of $y$ for large $\epsilon$.

\subsection{\b{The inverse Laplace transform of $s(x;\epsilon)$ in $x$}}\label{pros}
We first assume $V$ is $m+2$ times differentiable and
$$V(x)=
 \begin{cases}
    v_1 x^{-m}, & x\geqslant x_+>1 \\
    v_2 x^{-m}, & x\leqslant x_-<-1 \\
 \end{cases}
$$
By rescaling $x$ and $\epsilon$ one can make $v_{1,2}=\pm 1$.

To analyze the behavior of $s(x;\epsilon)$ for small $\epsilon$ and $x\geqslant x_+$, it is convenient to study its inverse Laplace transform, this time in $x$, to regularize the behavior at turning points. Inverse Laplace transformability
does not need to be proved at this stage, since at the end we show  that the Laplace transform of the
solution to the dual equation solves (\ref{eq:eqs}). We
let
$$\hat{s}(q)=:H(q)/[q(q+2\epsilon)]$$ be the formal inverse Laplace transform of $s$ ($x\mapsto q$ with $\Re x>x_+$), and obtain
\begin{equation}
  \label{eq:eqL}
  H(q)=\frac{v_1}{(m-1)!}q^{m-1}+\mathcal{P}^m\frac{v_1H(q)}{q(q+2\epsilon)}
\end{equation}  where $\mathcal{P}F(q)=\int_0^q F(u)du$.  With the change of variable
$q=\epsilon{\tau} $, $H(q)=F({\tau})$, we obtain
\begin{equation}
  \label{eq:eq7}
  F({\tau})=\frac{v_1\epsilon^{m-1} {\tau}^{m-1}} {(m-1)!}+\epsilon^{m-2} \mathcal{P}^m \frac{v_1F({\tau})}{{\tau}({\tau}+2)}
\end{equation}
The singularity structure of $s(x;\epsilon)$ (cf. \eqref{eq:eqs}) for small $\epsilon$ depends on the behavior of $F({\tau})$
for large $\tau$ in the complex plane. Let
$$\HH:=\{z:\Re\,z>0\};\ \ \HH_1:=\{z\in \mathbb{C}:\arg z \in(-\pi/4,5\pi/4),z\ne0\}$$
In the following we show that  $F({\tau})$ has an asymptotic expansion in powers of $\epsilon$, $\tau^{-1}$ and $\tau^{2-m}\ln \tau$ for small $\epsilon$ and large $\tau$.
\begin{Lemma}\label{genm3}
(i) The function $\hat{s}$ has a Laplace transform in $q$,  $F$ is analytic in $\tau$ for $\Re\, \tau>-2$ and entire in $\epsilon$ and has the convergent expansion
\begin{equation}
  \label{eq:ser2m}
  F(\tau)=\epsilon^{m-1}F_{m-1}({\tau})+ \sum_{j\geqslant m}\epsilon^{j_m} F_j({\tau})
\end{equation}
where $j_m=(m - 2) (j - m + 1) + m - 1$, $F_{m-1}(\tau)=v_1\tau^{m-1}/(m-1)!$,
while for $j\geqslant m$ we have
\begin{equation}
  \label{eq:eqFj1}
 F_j(\tau)=\tau^{j_m}\sum_{n=0}^{j-m+1}\tau^{-n(m-2)}(\ln(\tau))^nW_n(\tau^{-1});\ \ \text{\rm $W_n(z)$  analytic for $|z|<\frac12$}
  \end{equation}
 Furthermore, $F_j$ is analytic in $\HH_1$ and
      \begin{equation}
  \label{eq:bdF0}
|F_j|\lesssim  \frac{{|\tau|}^{j_m}}{((j_m-m+1)!)^{\frac{m}{m-2}}}\  \text{ for }\  \tau \in \HH_1
\end{equation}

(ii) For $|\tau|\ge 3$ we have\footnote{In fact we only need to keep the first few terms of \eqref{eq:eqFj1} and estimate the remainder.}
\begin{multline}
  \label{eq:estcoeffs}
  F_j(\tau)=F_j^{[0]}(\tau)+\tau^{j_m-m} G_{j}(\tau);\ \\
F_j^{[0]}(\tau):=\tau^{j_m}\bigg(\sum_{n=0}^{m-1}a^{[j]}_{n,0}\tau^{-n}+ \sum_{k=1}^2 \sum_{n=0}^{2-k} a^{[j]}_{k(m-2)+n,k}\tau^{-k(m-2)-n}(\ln\tau)^k \bigg);\  \\
\end{multline}
where  $a^{[j]}_{l,k}=0$ if $j<m+k-1$ or $l>m$, and the following estimates hold: for some $c_1>0$ (independent of all indices above)
\begin{equation}
  \label{eq:estcoeff}
 |a^{[j]}_{n,l}|\leqslant \frac{c_1^n}{((j_m-2m+2)!)^{\frac{m}{m-2}}},\ \  \sup_{\tau\in\HH_1}|(|\ln \tau|+1)^{-3}G_{j}(\tau)|\leqslant \frac{c_1 j}{((j_m-2m+1)!)^{\frac{m}{m-2}}}
\end{equation}

\end{Lemma}

\begin{proof}
 We analyze the case  $v_1=1$; if $v_1=-1$, the arguments are very similar.
We look for solutions to \eqref{eq:eq7} which are $O(\epsilon^{m-1} {\tau}^{m-1})$ for small ${\tau}$. Consider
the space $\mathcal{B}$ of functions of the form $f({\tau})={\tau}^{m-1} \tilde{G}({\tau})$ where $\tilde{G}$ is analytic
for, say,  $|{\tau}|<\tau_0$  for arbitrarily large $\tau_0>0$ with the norm  $\|f\|=\sup\{|\tilde{G}({\tau})|:\ |\tau|<\tau_0;\tau>-a>-2\}$.
We see
that this is a Banach space, and eq. (\ref{eq:eq7}) is contractive
in $\mathcal{B}$. The solution
of (\ref{eq:eq7}) is unique, and it is analytic for small ${\tau}$. As a differential equation (\ref{eq:eq7}) reads
\begin{equation}
  \label{eq:eqdiff}
  F^{(m)}=\frac{\epsilon^{m-2} F}{{\tau}({\tau}+2)}
\end{equation}
The argument above, or Frobenius theory, shows that (\ref{eq:eqdiff}) also has
a unique solution which is of the form $ \frac{1}{(m-1)!}\epsilon^{m-1} {\tau}^{m-1}(1+o(1))$ for small ${\tau}$. The solution is obviously analytic for
$\Re\, {\tau}>-2$, since the only singularities of equation \eqref{eq:eq7} are
${\tau}=0$ and ${\tau}=-2$, and it is entire in $\epsilon$ for $\Re\,
{\tau}>-2$, since the equation depends analytically on $\epsilon$.

By standard ODE asymptotic results \cite{Wasow} we see that any
solution of (\ref{eq:eqdiff}) is uniformly bounded in $\CC$ by
\begin{equation}
  \label{eq:eqC2}
 C(\epsilon) |\tau|^{\frac{m-1}{m}}e^{ \frac{m}{m-2}|\epsilon|^{1/m}|\tau|^{1-2/m}}
\end{equation}
for some $C(\epsilon)>0$. This ensures the necessary
(sub)exponential bounds for taking the Laplace transform in $\epsilon$.

We now look for solutions of (\ref{eq:eqdiff}) in the
form
\begin{equation}
  \label{eq:ser2}
  F=\frac{\epsilon^{m-1} {\tau}^{m-1}} {(m-1)!}  +\sum_{j\geqslant m}\epsilon^{j_m} F_j
\end{equation}
and we show that the expansion \eqref{eq:ser2} is convergent.

The functions $F_j$ satisfy the recurrence
\begin{equation}
  \label{eq:reci}
 F_{j+1}=\mathcal{P}^m  \frac{F_j}{{\tau}({\tau}+2)}, \ j\geqslant  m-1;\ \ F_{m-1}(\tau)=\tau^{m-1}/(m-1)!
\end{equation}
For now, we take ${\tau}$ in $\mathbb{H}_1$.
 It can be checked by induction that the $F_j$s are analytic in $\mathbb{H}_1$
and at zero, and since $|\tau^{j_m}|\leqslant |\tau^{j_m-1}(\tau+2)|$ in $\mathbb{H}_1$ and
$$|\mathcal{P}^m \tau^{j_m-2}|= \frac{|\tau|^{j_m+m-2}}{\prod_{k=0}^{m-1} (j_m-1+k)}\leqslant
\(\frac{(j_m-m+1)!}{((j+1)_m-m+1)!}\)^{\frac{m}{m-2}} |\tau|^{(j+1)_m}$$
\eqref{eq:bdF0} follows by induction. The last inequality above comes from the fact that
\begin{equation}
\label{ineq1}
\(\prod_{k=0}^{m-1} (j_m-1+k) \)^{m-2}\!\!\!\!\!\geqslant \(j_m-1 \)^{m(m-2)} \geqslant  \(\prod_{k=0}^{m-3} (j_m-m+2+k) \)^{m}
\end{equation}

It follows that the series (\ref{eq:ser2}) converges uniformly
on any compact set in $\mathbb{H}_1$. Moreover, we see
that
the function series
\begin{equation}
  \label{eq:ser2h}
 H(q)=\frac{ q^{m-1}} {(m-1)!}+\sum_{j\geqslant  m}\epsilon^{j_m} F_j(q/\epsilon)
\end{equation}
also converges uniformly
in any compact set in $\HH_1$.
Existence of the Laplace transform of $\frac{H(q)}{q(q+2\epsilon)}$ follows
from the bound \eqref{eq:eqC2} for $F$.

 We write \eqref{eq:eq7} as
\begin{equation}
  \label{eq:recd}
\tau(\tau+2) F_{j+1}^{(m)}=F_j, \ j\geqslant m-2;\ \ \ F_{m-2}=0
\end{equation}
Note that $F_{m-1}$ is explicit (see \eqref{eq:reci}).
Let $Lg=\tau(\tau+2)g^{(m)}(\tau)$. Eq. \eqref{eq:recd} implies
\begin{equation}
  \label{eq:eqL0}
  L^{j-m+2} F_j=0;\ \ \ j\geqslant  m
\end{equation}
\begin{Note}
  The indicial polynomial of \eqref{eq:eqL0} at infinity is
\begin{equation}
  \label{eq:eqlambda}
  \prod_{n=0}^{j-(m-1)}\prod_{n'=0}^{m-1}(\lambda-n'-n), \,\,j\geqslant m
\end{equation}
with the convention that a product is one  if the index set is empty,
and, by Frobenius theory,  \eqref{eq:eqlambda} implies \eqref{eq:eqFj1}.
 Eq. \eqref{eq:eqlambda} follows from
 \begin{equation}
    \label{eq:ind1}
    L \tau^\lambda=\tau^{\lambda-m+2}[\lambda(\lambda-1)\cdots(\lambda-m+1)+O(\tau^{-1})]
  \end{equation}
\end{Note}

(ii) The existence of an asymptotic expansion of the form \eqref{eq:estcoeffs} follows from \eqref{eq:eqFj1}. It remains to estimate the coefficients and the remainder (which we do recursively), for which we can assume $j_m\geqslant 2m-2$ since for $j_m<2m-2$ the result follows directly from \eqref{eq:eqFj1}.

 We have
\begin{equation}
  \label{eq:eqrec2}
  L(t^n\ln^l(t))=t^{n-1}\ln^lt (1+2t^{-1})\bigg(n(n-1)\cdots(n-m+1)+\sum_{l=1}^mP_l(\ln t)^{-l}\bigg)
\end{equation}
where $P_l$ are polynomials of degree at most $m-1$ in $n$ and $m$ in $l$. Substituting \eqref{eq:ser2m} in \eqref{eq:recd} using the notation in \eqref{eq:estcoeffs} and taking $
a^{[j]}_{n,l}= ((j_m-2m+2)!)^{-\frac{m}{m-2}}A^{[j]}_{n,l}$
we get the following recurrence for $0\leqslant  l\leqslant  3$ with
$(m-2)l\leqslant  n\leqslant  m-1$
\begin{equation}
\label{eq:recA}
A^{[j]}_{n,l}-C^{[j]}_{n,l} A^{[j-1]}_{n,l}+\sum_{J_{n,l}} C^{[j]}_{n',l';n,l}A^{[j]}_{n',l'}=0
\end{equation}
where $J_{n,l}$ consists of indices earlier than $n,l$:
$J_{n,l}=\{(n,l'):l'>l\}\bigcup \{(n-1,l'):0\leqslant  l'\leqslant  2\}$. In \eqref{eq:recA} we have $0<C^{[j]}_{n,l}<1$ and
for some $c_4>0$ and all $n,n',l,l',j$ we have $|C^{[j]}_{n',l';n,l}|<c_4$. Solving for
$A^{[j]}_{n,l}$ in the order $n=0,1,...,m-1$ and for a fixed $n$ in the
order $l=2,1,0$, the first inequality in \eqref{eq:estcoeff} follows
inductively on $j$.

 \begin{Note}{\rm  Let $R_j=\tau^{j_m-m-1}G_j$.  Then, $R_j$ satisfies the recurrence
  \begin{equation}
    \label{eq:eqrem}
    \tau(\tau+2)R_j^{(m)}=R_{j-1}+\tau^{j_m-2m+1}p_1(\ln \tau)
  \end{equation}
where $p_1$ is a quadratic polynomial with coefficients bounded by $$\ds
\frac{c_2}{[((j-1)_m-2m+2)!]^{\frac{m}{m-2}}}$$ for some $j-$independent  $c_2$.}

  Equation \eqref{eq:eqrem} simply follows by writing  $F_j(\tau)=F_j^{[0]}+R_j(\tau)$,  calculating the finite sum $LF_j^{[0]}$ explicitly using \eqref{eq:ind1} and \eqref{eq:eqrec2} and estimating the coefficients of $p_1$ using the first inequality in \eqref{eq:estcoeff}.
\end{Note}

\z {\it Proof of the last inequality in \eqref{eq:estcoeff}}. Since $j_m=(j-1)_m+m-2$, we have by  \eqref{eq:eqrem}
\begin{multline}\label{eq:eqrem1}
|R_j|=|\mathcal{P}^m \frac{\tau^{(j-1)_m-m-1}( G_{j-1}(\tau)+p_1(\ln \tau))}{\tau+2}|
\le \mathcal{P}^m|\tau^{(j-1)_m-m-2}(\tau G_{j-1}(\tau)+p_1(\ln \tau)) |
  \end{multline}

 Also, by direct integration we have
\begin{equation}
  \label{eq:eqln012}
 \int_0^\tau |t^n\ln^l t|dt\leqslant  c_3(n+1)^{-1} |\tau^{n+1}|(|\ln^l \tau|+1);\ \ l=0,1,2; \ n>0
\end{equation}
for some $c_3$. The rest  follows from \eqref{eq:eqrem1} and \eqref{eq:eqln012} by induction on $j$,
noting that (cf. also \eqref{ineq1})
\begin{equation}
  \label{eq:factest}
  \frac{1}{\prod_{k=0}^{m-1}((j-1)_m-m-1+k)} \le \(\frac{((j-1)_m-2m+1)!}{(j_m-2m+1)!}\)^{\frac{m}{m-2}}
\end{equation}
\end{proof}
\z Since $\hat{s}(q)=H(q)/[q(q+2\epsilon)]$ and $H(q)=F(q/\epsilon)$, the expansions for $F$ in Lemma \ref{genm3} allows us to obtain the corresponding  expansions for $H$ and $\hat{s}$:
\begin{Lemma}\label{hqexp}
(i) For $|q|\geqslant 3|\epsilon|$ \footnote{$3$ can be replaced by any constant
  bigger than $2$.} we have the expansion
\begin{multline}\label{eq:estH4}
  H(q)=q^{m-1}H_{0,0}(q)+\sum_{k=1}^{m-1}\epsilon^{k}q^{2m-3-k}H_{k,0}(q)+\sum_{n=1}^{2} \sum_{k=0}^{2-n}\epsilon^{n(m-2)+k} q^{m-1-k}H_{k,n}(q)\ln^n(q/\epsilon)\\+\epsilon^{m}\tilde{R}(q/\epsilon,q)
\end{multline}
where $H_{i,j}(q)$ ($j\leqslant  2$) are analytic in $q$ with sub-exponential growth for large $q$, $H_{0,2}=0$ if $m>3$, and $|{\partial^{(k+l)} \tilde{R}(u,v)}/{\partial u^k\partial v^l}|\lesssim (|\ln u|+1)^3|u|^{-k}$ for $\Re\, u\geqslant 0$, $q/\epsilon\in\mathbb{H}_1$,  $|v|<const.$, and $0\leqslant  k+l \leqslant  m+1$.

(ii) For $|q|\leqslant  3|\epsilon|$ and $\Re\, (q/\epsilon)\geqslant 0$ $H(q)$ is analytic in $q$, entire in $\epsilon$, and $|H(q)|\lesssim |q|^{m-1}$.
\end{Lemma}
\begin{proof}
  (i) Recall that $q=\epsilon{\tau}$, $H(q)=F(\tau)$, and $a^{[j]}_{l,k}=0$ if $j<m+k-1$ by Lemma \ref{genm3}. We thus substitute $\tau=q/\epsilon$ in \eqref{eq:estcoeffs} and obtain \eqref{eq:estH4} by collecting coefficients of powers of $\epsilon$ and $\ln(q/\epsilon)$. We define
  $$q^{m-1}H_{0,0}(q)=\frac{q^{m-1}}{(m-1)!}+\sum_{j\geqslant m}q^{j_m}a^{[j]}_{0,0}$$
  $$q^{2m-3-k}H_{k,0}(q)=\sum_{j\geqslant m}a^{[j]}_{k,0}q^{j_m-k} ~(1\leqslant k\leqslant m-1)$$
  $$ q^{m-1-k}H_{k,n}(q)= \sum_{j\geqslant m+n-1} a^{[j]}_{n(m-2)+k,n} q^{j_m-(n(m-2)+k)} ~(1\leqslant n\leqslant 2, 0\leqslant k\leqslant 2-n) $$
 $$ \tilde{R}(q/\epsilon,q)=\sum_{j_m\geqslant m}q^{j_m-m} G_j(q/\epsilon)$$
Convergence of the three series is ensured by \eqref{eq:estcoeff}.
Note that the $H_{k,n}$ are analytic in $q$ for all choices of $k$ and $n$ above.  Since  $F_j(\tau)$ is analytic in $\HH_1$ (see lemma \ref{genm3}), by \eqref{eq:estcoeff} and Cauchy's formula we have $$|G_j^{(k)}(\tau)|\lesssim \frac{ j|\tau|^{-k}(|\ln\tau|+1)^3}{((j_m-2m+1)!)^{\frac{m}{m-2}}}$$
for $\tau \in \HH$.
 Noting that for $a>0$ we have
$(j!)^a>const.^j \Gamma(aj+1)$, \eqref{eq:estcoeff} implies
  $$\sum_{j=m}^\infty\frac{|\tau|^{j_m}}{((j_m-2m+2)!)^{\frac{m}{m-2}}}\leqslant
  \sum_{j=0}^\infty\frac{|\tau|^{j}}{(j!)^{\frac{m}{m-2}}}\lesssim
  \sum_{j=0}^\infty\frac{|c_t\tau|^{j}}{\Gamma(\frac{m}{m-2}j+1)} \lesssim
   e^{\delta |\tau|}, \ \ \forall \delta>0
  $$
 Thus sub-exponential growth of $H_{i,j}$ follows. In fact it is elementary to show that the last sum is bounded by $e^{const.\tau^{\frac{m-2}{m}}}$.The rest of (i) follows from \eqref{eq:ser2m} using \eqref{eq:bdF0} and \eqref{eq:estcoeffs} to estimate the terms.

(ii) This follows from \eqref{eq:ser2m} and \eqref{eq:bdF0}.
\end{proof}

\begin{Corollary}\label{coro1} (i) For $|q|\geqslant 3|\epsilon|$ we have
\begin{equation}\label{coro1e}
\hat{s}(q)=\frac{H(q)}{q(q+2\epsilon)}=\sum_{n=0}^2 \epsilon^{n(m-2)}\(\frac{\epsilon^{m-2} \tilde{H}_{1,n+1}(\epsilon)}{q+2\epsilon}+\tilde{H}_{2,n+1}(q,\epsilon)\)\ln ^{n}(q/\epsilon)
+\epsilon^{m}\frac{\tilde{R}(q/\epsilon,q)}{q(q+2\epsilon)}
\end{equation}
where $\tilde{H}_{1,j}$ are entire, $\tilde{H}_{2,j}(q,\epsilon)$ are entire in $\epsilon$, analytic in $q$ and have sub-exponential growth for large $q$, $\tilde{H}_{2,1}(\epsilon,\epsilon)=O(\epsilon^{m-3})$, $\tilde{H}_{1,1}(0)=\ds \frac{(-2)^{m-2}}{\Gamma(m)}$, $\tilde{H}_{k,3}=0$ if $m>3$, and $\tilde{R}$ is the same as in Lemma \ref{hqexp}.

(ii) For $|q|\leqslant  3|\epsilon|$ and $\Re\, (q/\epsilon)\geqslant 0$ we have
\begin{equation}
  \label{eq:dH1}
  \hat{s}(q)=\frac{H(q)}{q(q+2\epsilon)}= \frac{\epsilon^{m-2}H_1(\epsilon)}{q+2\epsilon}+\epsilon^{m-3}H_2(q,\epsilon)
\end{equation}
where $H_i$ are analytic in $q$ and entire in $\epsilon$.
\end{Corollary}
\begin{proof}
Note that for a  function $f$ analytic in $q$ and entire in $\epsilon$ we have
\begin{equation}\label{fq1}
\frac{f(q;\epsilon)}{q+2\epsilon}=\frac{f(-2\epsilon)}{q+2\epsilon}+\frac{f(q)-f(-2\epsilon)}{q+2\epsilon}:=\frac{\tilde{f}_1(\epsilon)}{q+2\epsilon}+\tilde{f}_2(q;\epsilon)
\end{equation}
where $\tilde{f}_2$ is analytic in $q$ and entire in $\epsilon$. Using now \eqref{eq:estH4}, we take
$$qf(q;\epsilon)=q^{m-1}H_{0,0}(q)+\sum_{k=1}^{m-1}\epsilon^{k}q^{2m-3-k}H_{k,0}(q)$$
in \eqref{fq1} and define $\epsilon^{(2-k)(m-2)}\tilde{H}_{k,1}=\tilde{f}_k ~ (k=1,2)$. For $k,n=1,2$ we take
$$qf(q;\epsilon)=\sum_{k=0}^{2-n}\epsilon^{k} q^{m-1-k}H_{k,n}(q)$$
in \eqref{fq1} and define $\epsilon^{(2-k+n)(m-2)}\tilde{H}_{k,n+1}=\tilde{f}_k ~ (k,n=1,2)$.
Sub-exponential growth of $\tilde{H}_{2,j}$ follows immediately from the sub-exponential growth of $H_{2,j}(q)$. In addition   $\tilde{H}_{1,1}(0)=H_{0,0}(0)=-\frac{1}{2}F_{m-1}(-2)=\ds \frac{(-2)^{m-2}}{\Gamma(m)}$ by Lemma \ref{genm3} and the proof of Lemma \ref{hqexp}.

Similarly to obtain \eqref{eq:dH1} we use \eqref{fq1} for $f(q)=H(q)$ and apply Lemma \ref{hqexp} (ii).
\end{proof}

\subsection{Asymptotic expansion of $s(x;\epsilon)$ for small $\epsilon$}\label{smallepss}
For small $\epsilon$, the  function $s(x;\epsilon)$ (cf. \eqref{eq:defs}) has an asymptotic
expansion in powers of $\epsilon$ and $\epsilon^{m-2}\ln\epsilon$. To obtain the time decay of $\psi$, only a few terms of this expansion are needed, and they are obtained in Lemma \ref{coro} below.

\begin{Lemma}
\label{coro}
(i) Let $\delta>0$ be arbitrarily small but fixed.
We have for $x\geqslant x_+$, $\epsilon \in \HH$ and $|\epsilon|\leqslant 1/x$
\begin{equation}\label{sm4}
s(x;\epsilon)=h_{1}(x)\epsilon^{m-2}\ln\epsilon+h_{2}(x)\epsilon^{m-1}\ln\epsilon+h_{3}(x)\epsilon^{2m-4}(\ln\epsilon)^2
+Q(x;\epsilon)
\end{equation}
where the smooth functions $h_{j}$ satisfy \begin{equation}\label{sm4a}
 h_1(x)\sim- \frac{(-2)^{m-2}}{\Gamma(m)},h_2(x)\sim \frac{(-2)^{m-1}x}{\Gamma(m)},h_3(x)\sim a_0~\text{as}~ x\to\infty\end{equation}
 where $a_0$ is a constant and for $m>3$ one can take $h_{3}(x)=0$. Furthermore

\begin{equation}\label{unibd}
\ds \sup_{0\leqslant  k\leqslant  m-1}\left|x^{m-2-k} \frac{\partial^k Q(x;\epsilon)}{\partial \epsilon^k}\right|<\infty ; \ \ \sup_{m\leqslant  k\leqslant  m+1}\left|\epsilon^{k-m+\delta}x^{-2} \frac{\partial^k Q(x;\epsilon)}{\partial \epsilon^k}\right|<\infty
\end{equation}

The asymptotic formula \eqref{sm4} is twice differentiable in $x$, i.e. \eqref{unibd} holds with $Q$ replaced by $xQ'$ or $x^2Q''$.
\end{Lemma}

\begin{proof}
We write
\begin{equation}\label{threeint}
s(x;\epsilon)=\(\int_0^{3\epsilon} +\int_{3\epsilon}^{1} +\int_1^{\infty}\) \frac{H(q)e^{-q x}}{q(q+2\epsilon)}  dq
\end{equation}
The first and last terms are estimated easily: the last integral is manifestly analytic in $\epsilon$ and decays exponentially in $x$, so \eqref{unibd} is obvious. By Corollary \ref{coro1} the first integral is equal to
$$\epsilon^{m-2}\int_{0}^3\(\frac{H_1(\epsilon)}{\tau+2}+H_2(\epsilon\tau,\epsilon)\)e^{-\epsilon\tau x} d\tau$$
Thus it is analytic in $\epsilon$ and satisfies the estimates in \eqref{unibd} by direct differentiation in $\epsilon$, noting that $|\epsilon^{m-2}|\lesssim x^{2-m}$.

To estimate the middle integral in \eqref{threeint} we use \eqref{coro1e} to write
$$\int_{3\epsilon}^{1} \frac{H(q)e^{-q x}}{q(q+2\epsilon)}  dq=S_1(x,\epsilon)+S_2(x,\epsilon)+S_3(x,\epsilon)$$
where
\begin{equation}\label{defsa}
S_1(x,\epsilon)=\sum_{k=1}^3\epsilon^{k(m-2)}\tilde{H}_{1,k}(\epsilon)\int_{3\epsilon}^1\frac{(\ln (q/\epsilon))^{k-1}e^{-q x}}{q+2\epsilon}dq
\end{equation}
\begin{equation}\label{defsb}
S_2(x,\epsilon)=\sum_{k=0}^2\epsilon^{k(m-2)}\int_{3\epsilon}^1(\ln (q/\epsilon))^{k}\tilde{H}_{2,k+1}(q,\epsilon)e^{-q x}dq
\end{equation}
\begin{equation}\label{defsc}
S_3(x,\epsilon)=\epsilon^{m} \int_{3\epsilon}^1 \frac{\tilde{R}(q/\epsilon,q)e^{-q x}}{q(q+2\epsilon)}dq
\end{equation}
Now the fact that the middle integral in \eqref{threeint} has an expansion of the form \eqref{sm4} follows from the lemma below:
\begin{Lemma}\label{lem5}
The terms $S_1(x,\epsilon)$ and $S_2(x,\epsilon)$ have expansions of the form \eqref{sm4}, and $S_3(x,\epsilon)$ satisfies the estimates in \eqref{unibd}. More precisely we have
\begin{equation}\label{newad1}
S_k(x;\epsilon)=h_{k,1}(x)\epsilon^{m-2}\ln\epsilon+h_{k,2}(x)\epsilon^{m-1}\ln\epsilon+h_{k,3}(x)\epsilon^{2m-4}(\ln\epsilon)^2
+R_{k,0}(x;\epsilon)
\end{equation}
for $k=1,2$,
where $h_{k,j}$ have the large $x$ asymptotics
$$h_{1,1}\sim -\frac{(-2)^{m-2}}{\Gamma(m)}(1+o(1)); \ h_{1,2}\sim-\frac{(-2)^{m-2}(2x+o(x))}{\Gamma(m)}
; \ h_{1,3}\sim a_0+o(1)$$
$$h_{2,1}=o(1); \ h_{2,2}=O(1)
; \ h_{2,3}=o(1)$$
where $a_0$ is a constant and $R_{k,0}$ satisfy the estimates in \eqref{unibd}.
\end{Lemma}
\begin{proof}
We first show the result for $S_1$ using \eqref{defsa}. Since $\tilde{H}_{1,k}$ are analytic in $\epsilon$, we only need to analyze the integrals
\begin{equation}\label{taude0}
\int_{3\epsilon}^1\frac{\ln ^l(q/\epsilon)e^{-q x}}{q+2\epsilon}dq=-\int_{1}^{\infty}\frac{\ln ^l(q/\epsilon)e^{-q x}}{q+2\epsilon}dq+\int_{3}^{\infty}\frac{\ln ^l(\tau)e^{-\epsilon\tau x}}{\tau+2}d\tau
\end{equation}
where $0\leqslant l\leqslant 2$. The first integral on the right hand side is a polynomial in $\ln \epsilon$ times a function analytic in $\epsilon$ with exponential decay in $x$, and the last one is equal to
\begin{equation}\label{taude}
\int_{3}^{\infty}\ln ^l\tau(\tau^{-1}-2\tau^{-2})e^{-\epsilon\tau x}d\tau +\int_{3}^{\infty}\frac{4\ln ^l\tau}{\tau^2(\tau+2)}e^{-\epsilon\tau x}d\tau
\end{equation}
To analyze the first term in \eqref{taude} we need the following elementary result:
\begin{Lemma}\label{nm0} Assume $l \geqslant 0$, $x\geqslant x_+$ and $|\epsilon x|\leqslant 1$. For $n \geqslant 0$ we have
\begin{equation}\label{inttau}
\int_3^{\infty}e^{-\epsilon \tau x} \tau^n(\ln \tau)^ld\tau=
\frac{1}{(\epsilon x)^{n+1}}\sum_{q=0}^{l}c_{q}^{(n;l)}\ln ^q (\epsilon x) +R_{a,n}(x,\epsilon)
\end{equation}
where $c_{q}^{(n;l)}$ are constants with $c_{1}^{(0;1)}=-1$, and $R_{a,n}$ satisfies
\begin{equation}\label{unibd1}
\left|\frac{\partial^k R_{a,n}(x,\epsilon)}{\partial \epsilon^k}\right|\lesssim x^k \ \ (k\geqslant 0)
\end{equation}
In addition, for $n\leqslant  -1$ we have
\begin{equation}\label{nm1}
\int_3^{\infty}e^{-\epsilon \tau x} \tau^{n}(\ln \tau)^ld\tau=(\epsilon x)^{-1-n}\sum_{q=0}^{l+1}c_{q}^{(n;l)}  \ln^q(\epsilon x) +R_{a,n}(x,\epsilon)
\end{equation}
where $c_{1}^{(-1;0)}=-1$, $c_{1}^{(-2;0)}=1$, and $R_{a,n}$ satisfies \eqref{unibd1}.
\end{Lemma}
The proof of this Lemma is given in the Appendix.

Expansion of the first term in \eqref{taude} follows directly from Lemma \ref{nm0}. The last term in \eqref{taude}  satisfies
\begin{multline}\label{taude1}
\left|\frac{d^k}{d \epsilon^k}\int_{3}^{\infty}\frac{4\ln
    ^l\tau}{\tau^2(\tau+2)}e^{-\epsilon\tau x}d\tau \right|
\lesssim \left|x^k\int_{3}^{\infty}\tau^{-3+k}\ln ^l\tau e^{-\epsilon\tau x}d\tau\right| \\
\lesssim
\left\{
  \begin{array}{ll}
 x^k, & \hbox{$0\leqslant  k\leqslant  1$;} \\
 x^2|\epsilon|^{2-k}(1+ |\ln ^{l+1}(\epsilon x)|), & \hbox{$2\leqslant  k\leqslant  m+1$.}
  \end{array}
\right.
\end{multline}
where the second integral in \eqref{taude1} is estimated using Lemma \ref{nm0}. Thus
$$\epsilon^{(l+1)(m-2)+k}\int_{3}^{\infty}\frac{4\ln ^l\tau}{\tau^2(\tau+2)}e^{-\epsilon\tau x}d\tau$$
satisfies \eqref{unibd} for all $k\geqslant 0$ and $l\geqslant 0$ by direct differentiation.

Thus combining \eqref{taude0} and \eqref{taude} using Lemma \ref{nm0} we see that
\begin{multline}\label{gene1}\epsilon^{(l+1)(m-2)}\int_{3\epsilon}^1\frac{\ln ^l(q/\epsilon)e^{-q x}}{q+2\epsilon}dq= \epsilon^{(l+1)(m-2)}\ln ^l\epsilon\int_{1}^{\infty}\frac{e^{-q x}}{q+2\epsilon}dq \\
+\sum_{n=-2}^{-1}(-2)^{-1-n}\epsilon^{(l+1)(m-2)-1-n}x^{-1-n}\sum_{q=0}^{l+1}c_{q}^{(n;l)}  \ln^q(\epsilon x)+
R_{0,0}(x,\epsilon)\end{multline}
where $R_{0,0}$ satisfies \eqref{unibd}.

Letting $l=0,1,2$ in \eqref{gene1} we have
\begin{equation}\label{esto1}
\epsilon^{m-2}\tilde{H}_{1,k}(\epsilon)\int_{3\epsilon}^1\frac{e^{-q x}}{q+2\epsilon}dq=h_{a,1}(x)\epsilon^{m-2}\ln\epsilon+h_{a,2}(x)\epsilon^{m-1}\ln\epsilon+R_{0,1}(x,\epsilon)
\end{equation}
where $h_{a,1}(x)=-1+o(1)$ and $h_{a,2}(x)=-2x(1+o(1))$ for large $x$,
\begin{equation}\label{esto2}
\epsilon^{2m-4}\tilde{H}_{2,k}(\epsilon)\int_{3\epsilon}^1\frac{\ln (q/\epsilon)e^{-q x}}{q+2\epsilon}dq=c_{2}^{(-1;1)}\epsilon^{2m-4}\ln^2\epsilon
+h_{b,1}(x)\epsilon^{2m-4}\ln\epsilon+R_{0,2}(x,\epsilon)
\end{equation}
where $h_{b,1}(x)=O(1)$ for large $x$, and
\begin{equation}\label{esto3}
\epsilon^{3m-6}\tilde{H}_{3,k}(\epsilon)\int_{3\epsilon}^1\frac{\ln^2 (q/\epsilon)e^{-q x}}{q+2\epsilon}dq=R_{0,3}(x,\epsilon)
\end{equation}
where $R_{0,i} ~ (0\leqslant i\leqslant 3)$ satisfies \eqref{unibd}. Thus \eqref{newad1} follows from \eqref{esto1}, \eqref{esto2}, and \eqref{esto3}. Note that $\tilde{H}_{1,1}(0)=\ds \frac{(-2)^{m-2}}{\Gamma(m)}$ by Corollary \ref{coro1}.

To show \eqref{newad1} for $S_2$ we write $\ln(q/\epsilon)=\ln q-\ln \epsilon$ in \eqref{defsb}. By Corollary \ref{coro1}
$$ \int_{2\epsilon}^{1} \tilde{H}_{2,n}(q,\epsilon)e^{-q x}dq$$
is entire in $\epsilon$ with its $k$-th derivative in $\epsilon$ bounded by $const.x^{k-1}$. Thus the term containing $\tilde{H}_{2,1}$ satisfies \eqref{unibd} and the terms containing $\ln \epsilon\tilde{H}_{2,2}$ and $\ln^2 \epsilon\tilde{H}_{2,3}$ satisfy \eqref{newad1}.

The term with $\ln q \tilde{H}_{2,2}$ is analyzed using integration by
parts:
 \begin{multline}\label{parts}
 \epsilon^{m-2} \int_{3\epsilon}^{1} \ln q \tilde{H}_{2,2}(q,\epsilon)e^{-q x}dq=
\epsilon^{m-2}\tilde{H}_{2,2}(q,\epsilon)e^{-q x}\(q\ln q-q\)\bigg|_{3\epsilon}^{1}
\\-\epsilon^{m-2}\int_{3\epsilon}^{1}\(q\ln q-q\)\frac{\partial(\tilde{H}_{2,2}(q,\epsilon)e^{-q x})}{\partial q }dq
 \end{multline}
where the last term satisfies \eqref{unibd} by direct
calculation and counting powers of $\ve$.

The term containing $\ln q\tilde{H}_{2,3}(q,\epsilon)$ can be similarly analyzed using integration by parts, which gives
$$\epsilon^{2m-4}\ln \epsilon \int_{3\epsilon}^{1} \ln q \tilde{H}_{2,3}(q,\epsilon)e^{-q x}dq=-\epsilon^{2m-4}\ln \epsilon\tilde{H}_{2,3}(1,\epsilon)e^{-x}+R_s(x,\epsilon)$$
where $R_s$ satisfies \eqref{unibd}.
The term containing $(\ln q)^2\tilde{H}_{2,3}(q,\epsilon)$ satisfies \eqref{unibd} by direct
calculation and counting powers of $\ve$.

Finally we show that $S_3(x,\epsilon)$ satisfies the estimate \eqref{unibd}. Denoting
$\partial_{(i,j)}\tilde{R}(u,v)=\partial^{i+j} \tilde{R}(u,v)/(\partial u)^i (\partial v)^j$, we have by Lemma \ref{hqexp}
 \begin{multline} \left|\frac{\partial^k }{\partial\epsilon ^k}\(\epsilon \int_{3\epsilon}^1 \frac{\tilde{R}(q/\epsilon,q)e^{-q x}}{q(q+2\epsilon)}dq\)\right|= \left|\frac{\partial^k }{\partial\epsilon ^k}\int_{3}^{1/\epsilon} \frac{\tilde{R}(\tau,\epsilon\tau)e^{-\epsilon\tau x}}{\tau(\tau+2)}d\tau\right|\\
 \lesssim \sum_{i+j=k} \left|\int_{3}^{1/\epsilon} \frac{\partial_{(0,i)}\tilde{R}(\tau,\epsilon\tau)\tau^{i-1}x^je^{-\epsilon \tau x}}{\tau+2}d\tau\right|
+\sum_{i+j=k-1}\left|\frac{\partial_{(i,0)}\tilde{R}(1/\epsilon,1)e^{- x}}{\epsilon^{2i+j}}\right|
\\ \lesssim (|\ln\epsilon|+1)^4\(|\epsilon^{1-k}|e^{- x}+\sum_{i+j=k}|\epsilon^{1-i}x^j|\)
 \end{multline}
for $0\leqslant  k \leqslant  m+1$. Thus
$$\left|\frac{\partial^n S_3(x,\epsilon)}{\partial\epsilon ^n} \right|
 \lesssim \sum_{j+k=n}\left|\epsilon^{m-1-j}\frac{\partial^k }{\partial\epsilon ^k}\(\epsilon \int_{3\epsilon}^1 \frac{\tilde{R}(q/\epsilon,q)e^{-q x}}{q(q+2\epsilon)}dq\)\right| \lesssim |(|\ln\epsilon|+1)^4\epsilon^{m-n}|$$

\end{proof}
Finally $s'$ and $s''$ are similarly analyzed, finishing the proof of \ref{coro}.
\end{proof}

\subsection{Detailed behavior of the exponentially decaying functions  $y_\pm$}\label{ypm}
Since $\hat{\psi}$ solves \eqref{psiwave}, its singularity structure depends on the singularity structure of the two solutions of the associated homogeneous equation \eqref{eq:hom1}, which are of the form $y_\pm(x)=e^{\mp\epsilon x}(1+s_\pm(x))$ where $s_\pm(x)=o(1)$ for $x\to \pm\infty$. By symmetry it is sufficient to study $y_+(x)$, where $s_+$ is exactly the function $s$ in Lemma \ref{coro}.

Proposition \ref{jostf} below shows that $y_+(x)$ has an expansion involving the special functions
 $$\Phi_1(x)=(m-2)^{1/(m-2)}\Gamma\(\frac{m-1}{m-2}\)\sqrt{x}I_{\frac{1}{m-2}}\(\frac{2x^{1-m/2}}{m-2}\)$$
$$\Phi_2(x)=\frac{(m-2)^{-1/(m-2)}}{\Gamma\(\frac{1}{m-2}\)}\sqrt{x}K_{\frac{1}{m-2}}\(\frac{2x^{1-m/2}}{m-2}\)$$
where $I_n$ and $K_n$ denote the modified Bessel functions of the first
and second kind respectively.

\begin{Proposition} \label{jostf}

For arbitrarily small $\delta>0$, $x\geqslant x_+$ and $|\epsilon|\leqslant  1/x$ we have
\begin{equation}
\label{cancel0}
y_+(x;\epsilon)=r(\epsilon) \bigg(
\Phi_1(x)+B_1(x)\epsilon^{m-1}\ln\epsilon+\hat{f}_a(x,\epsilon)\bigg)
\end{equation}
where \begin{equation}\label{defr}
r(\epsilon):=\left(1+a_1\epsilon^{m-2}\ln\epsilon+a_0\epsilon^{2m-4}(\ln\epsilon)^2\right)
\end{equation}
where $a_{0,1}$ are constants, $B_1(x)$ is a linear combination of $\Phi_{1,2}(x)$ with $B_{1}(x)\sim const.x$, $\hat{f}_{a}(x,\epsilon)\lesssim |\epsilon x|$, and
\begin{equation}\label{fhat}
\sup_{1\leqslant  k\leqslant  m-1}\left|x^{-k}\frac{\partial^k \hat{f}_a(x;\epsilon)}{\partial \epsilon^k}\right| <\infty ; \\ \sup_{m\leqslant  k\leqslant  m+1}\left|\epsilon^{k-1-m+\delta} x^{-m-1}\frac{\partial^k \hat{f}_a(x;\epsilon)}{\partial \epsilon^k}\right| <\infty
\end{equation}
Furthermore, the expansion \eqref{cancel0} is differentiable in $x$, i.e. \eqref{fhat} holds with $\hat{f}_a(x;\epsilon)$ replaced by $x\hat{f}_a'(x;\epsilon)$.
\end{Proposition}

\begin{proof}
It follows from Lemma \ref{coro} that $y_+=e^{-\epsilon x}(1+s)$ has the following expansion
\begin{multline}\label{h04}
y_+(x;\epsilon)=h_0(x)+h_{1}(x)\epsilon+h_{2}(x)\epsilon^{m-2}\ln\epsilon+h_{3}(x)\epsilon^{m-1}\ln\epsilon
+h_{4}(x)\epsilon^{2m-4}(\ln\epsilon)^2+\tilde{H}(x;\epsilon)
\end{multline}
where $h_0\sim 1$, $h_{2}\sim a_1$, $h_{4}\sim a_0$ for large $x$, and $\tilde{H}(x;\epsilon)$ satisfies \eqref{fhat} with $|\tilde{H}(x;\epsilon)|\lesssim |\epsilon^2 x^2|$. Plugging this expansion back into \eqref{eq:hom1} (recall that the asymptotics \eqref{h04} is differentiable by Lemma \ref{coro}) we have
\begin{multline}
0=\left(h_0''(x)-\frac{1}{x^m}h_0(x)\right)+\left(h_{1}''(x)-\frac{1}{x^m}h_{1}(x)\right)\epsilon
+\left(h_{2}''(x)-\frac{1}{x^m}h_{2}(x)\right)\epsilon^{m-2}\ln\epsilon+...\\
\end{multline}
Standard asymptotic arguments for $\epsilon\to 0$ show that all the coefficients above must be zero, and thus
$h_{i}$ satisfies the equation
\begin{equation}\label{eqf0}
f''(x)=x^{-m}f(x)
\end{equation}
The solutions to \eqref{eqf0} are exactly $\Phi_{1,2}$. Note that $\Phi_1(x)=1+o(1)$ and $\Phi_2(x)=x(1+o(1))$ for large $x$. Therefore $h_{0}(x)=\Phi_1(x)$, and $h_{2}(x)=a_1\Phi_1(x)$, $h_{4}(x)=a_0\Phi_1(x)$ since $h_0\sim 1$, $h_{2}\sim a_1$, and $h_{4}\sim a_0$.

Thus dividing \eqref{h04} by \eqref{defr} we obtain \eqref{cancel0} for some $B_1$. Substituting \eqref{cancel0} into
\eqref{eq:hom1} and examining the coefficients of $\epsilon^{m-1}\ln
\epsilon$ we see that $B_1$ satisfies
\eqref{eqf0} and is thus a linear combination of $\Phi_{1,2}$. The large $x$
behavior of $B_{1}$ follows from \eqref{sm4} and \eqref{sm4a}. Differentiability of \eqref{cancel0} follows from the last paragraph of Lemma
\ref{coro}.

\end{proof}

\begin{Note}\label{otherj}(i) A similar conclusion is true for $y_-$.

(ii) Without loss of generality we can assume $r(\epsilon)\ne 0$, since otherwise we can modify the definition of $r(\epsilon)$ by adding $\epsilon$ to it.
\end{Note}

\subsection{Estimating the exponentially decaying functions $y_{\pm}$ for large $\epsilon$}\label{yleps}
To prove the main results we use Proposition \ref{jostf} and the properties of the inverse Laplace transform of
$y_\pm(x;\epsilon)$ (see Section \ref{prothm}). These properties depend on the behavior of $y_\pm$ and $s_\pm$ together with its $\epsilon$ derivatives in $\HH$ for large $\epsilon$.  Let $\mathcal{D}=(x_+,\infty)\times
\overline{\mathbb{H}}$. First we prove the following lemma:

\begin{Lemma}\label{sph}
Assume $|V^{(k)}(x)|\lesssim x^{-m-k}$ for $x\geqslant x_+$ and $0\leqslant k\leqslant 2$, and $s(x;\epsilon)=s_+(x;\epsilon)$ solves \eqref{eq:eqs} with $s(x;\epsilon)=o(1)$ for $x\to\infty$. Then for any fixed $x_1\in \mathbb{R}$ the function $s(x;\cdot)$ is analytic in $\mathbb{H}$ and continuous in $\overline{\mathbb{H}}$ for $x\geqslant x_1$.
Moreover, we have
\begin{equation}
  \label{eq:eqests}
  |s(x;\epsilon)|\lesssim  (|\epsilon| \langle x\rangle+1)^{-1}\langle x\rangle^{-m+2}; \ \  |s'(x;\epsilon)|\lesssim  (|\epsilon| \langle x\rangle+1)^{-1}\langle x\rangle^{-m+1}
\end{equation}
uniformly in $\mathcal{D}$, and for $|\epsilon|\geqslant 1/\langle x\rangle$ we have
\begin{equation}
  \label{eq:estder1}
  \left|\frac{\partial^n s(x;\epsilon)}{\partial \epsilon^n} \right|\lesssim  |\epsilon|^{-1-n}\langle x\rangle^{1-m};\ \left|\frac{\partial^n s'(x;\epsilon)}{\partial \epsilon^n} \right|\lesssim  |\epsilon |^{-1-n}\langle x\rangle^{-m}; \ 0\leqslant  n\leqslant  m+1
\end{equation}
A similar conclusion is true for $s_-$, i.e. \eqref{eq:eqests} and \eqref{eq:estder1} hold for both $s_+$ and $s_-$.
\end{Lemma}
\begin{proof}
  We write \eqref{eq:eqs} in integral form:
  \begin{multline}
    \label{eq:ctr1}
    s(x)=\int_\infty^{x} \int_{\infty}^te^{-2\epsilon(t'-t)}V(t')dt'dt+\int_\infty^{x} \int_{\infty}^te^{-2\epsilon(t'-t)}V(t')s(t')dt'dt:=T_1(x;V)+L\,s
  \end{multline}
It is straightforward to check that $|T_1(x;V)|\lesssim x^{2-m}$, and by integration by parts,
\begin{equation}
  \label{eq:est4}
 |T_1(x;V)|\leqslant \frac{1}{2|\epsilon|}\left|\int_\infty^{x} \(V(t)-\int_{\infty}^te^{-2\epsilon(t'-t)}V'(t')dt'\)dt\right|\lesssim |\epsilon|^{-1}x^{1-m}
\end{equation}
Thus
\begin{equation}
  \label{eq:est41}
 |T_1(x;V)|\lesssim \min(x^{2-m},|\epsilon|^{-1}x^{1-m}) \lesssim (|\epsilon x|+1)^{-1}x^{2-m}
\end{equation}
uniformly in $x\geqslant x_+$ and $\epsilon\in\overline{\mathbb{H}}$.
We analyze \eqref{eq:ctr1} in the Banach space $\mathcal{B}$
of functions $f:\mathcal{D}\to \CC$,
such that $f(x,\cdot)$ is analytic in $\mathbb{H}$,
continuous in $\overline{\mathbb{H}}$, with the norm
\begin{equation}
  \label{eq:norm1}
 \|s\|=\sup_{(x,\epsilon)\in\mathcal{D}}|s(x,\epsilon)|
<\infty
\end{equation}
We see that $T_1\in\mathcal{B}$ and
\begin{equation}
  \label{eq:eqn}
  \|L\|\leqslant \left| \int_\infty^{x} \int_{\infty}^t V(t')dt'dt \right| \leqslant  const. x^{2-m}
\end{equation}
Thus if $x_0$ is sufficiently large then \eqref{eq:ctr1} is contractive in $\mathcal{B}$ and  has a unique solution for $x\geqslant x_0$. Then, the first estimate in \eqref{eq:eqests} is obtained by taking
$x_0$ sufficiently large and writing
\begin{equation}
  \label{eq:eqest42}
  |s(x,\epsilon)|\leqslant  \left| (1-L)^{-1} T_1 \right|
\end{equation}

We see, by direct differentiation of the rhs of \eqref{eq:ctr1} that
  \begin{equation}
    \label{eq68}
    s'(x;\epsilon)=\int_{\infty}^x e^{-2\epsilon(t'-x)}V(t')dt'+ \int_{\infty}^x e^{-2\epsilon(t'-x)}V(t')s(t')dt'
  \end{equation}
This implies
$$|s'(x;\epsilon)|\lesssim x^{1-m}$$
and by integration by parts and the first estimate in \eqref{eq:eqests}
\begin{equation}
  \label{eq:eqsd1}
 \left|s'(x,\epsilon)+\frac{1}{2\epsilon}\(V(x)-\int_{\infty}^xe^{-2\epsilon(t'-x)}V'(t')dt'\)\right|
\lesssim|\epsilon|^{-1}x^{2(1-m)}
\end{equation}
Thus
\begin{equation}
  \label{eq:eqs1}
  |s'(x,\epsilon)|\lesssim  \min(|\epsilon |^{-1}x^{-m},x^{1-m})\lesssim (|\epsilon x|+1)^{-1}x^{1-m}
\end{equation}
Rewriting \eqref{eq68} as
$$s'(x;\epsilon)=\int_{\infty}^0 e^{-2\epsilon t}V(t+x)dt+ \int_{\infty}^0 e^{-2\epsilon t}V(t+x)s(t+x)dt $$
we have by integration by parts
$$s{''}(x;\epsilon)=-\frac{1}{2\epsilon}\(V'(x)-\int_{\infty}^0 e^{-2\epsilon t}V{''}(t+x)dt\)+ \int_{\infty}^0 e^{-2\epsilon t}(V(t+x)s(t+x))'dt $$
which implies
\begin{equation}\label{spp}
|s''(x;\epsilon)|\lesssim  |\epsilon|^{-1}x^{-m-1}
\end{equation}
For the $\epsilon$ derivatives, the proof is by induction on $n$. The equation for $u=\frac{d}{d\epsilon}s(x,\epsilon)$ is
\begin{equation}
  \label{eq:eqdifs}
  u''-2\epsilon u'-V(x)u=2s'(x)
\end{equation}

By \eqref{eq:eqdifs} we have
\begin{equation}\label{ueq}
u(x)=T_1(x;2s')+L\,u
\end{equation}
(see \eqref{eq:ctr1}).
Using \eqref{eq:eqs1} and \eqref{spp} we obtain
\begin{equation}
  \label{}
 |T_1(x;2s')|\leqslant \frac{1}{|\epsilon|}\left|\int_\infty^{x} \(s'(t)-\int_{\infty}^te^{-2\epsilon(t'-t)}s''(t')dt'\)dt\right|\lesssim |\epsilon|^{-2}x^{1-m}
\end{equation}
Thus \eqref{eq:eqn} and \eqref{ueq} imply
$$|u(x,\epsilon)|\lesssim |\epsilon |^{-2}x^{1-m}$$
$$|u'(x,\epsilon)|=\left|2\int_{\infty}^x e^{-2\epsilon(t'-x)}s'(t')dt'+ \int_{\infty}^x e^{-2\epsilon(t'-x)}V(t')u(t')dt'\right|\lesssim |\epsilon |^{-2}x^{-m}$$
Taking $k$ $\epsilon-$derivatives of \eqref{eq:eqdifs} and letting $u_k(x,\epsilon)=\partial^k s(x,\epsilon)/\partial \epsilon^k$ we have
\begin{equation}
  \label{eq75}
  u_k''-2\epsilon u_k'-V(x)u_k=2ku_{k-1}'
\end{equation}
which gives by induction
$$|u_k(x,\epsilon)|\leqslant  \left| (1-L)^{-1} T_1(x;2ku_{k-1}') \right|\lesssim |\epsilon |^{-1-k}|x|^{1-m}$$
$$|u_k'(x,\epsilon)|\lesssim |\epsilon |^{-1-k}|x|^{-m}$$
Finally, for $x_1\leqslant x\leqslant x_0$ existence and analyticity of the solution follow from standard analytic dependence on parameters. Thus \eqref{eq:eqests} and \eqref{eq:estder1} only need to be verified for large $\epsilon$ since the $x$ dependence is irrelevant. Note that by \eqref{eq:eqs} we have
  \begin{equation}
    \label{eq:ctr1a}
    s(x)=\frac{1}{2\epsilon}\int_{\infty}^{x}(e^{-2\epsilon(t-x)}-1)V(t)(1+s(t))dt
  \end{equation}
which is contractive for large $\epsilon$ implying $|s(x)|\lesssim 1/|\epsilon|$. The result for $s'$ then follows from \eqref{eq68}. Similarly using \eqref{eq75} we see that
  \begin{equation}
    \label{eq:ctr1b}
    u_k(x)=\frac{1}{\epsilon}\int_{\infty}^{x}(e^{-2\epsilon(t-x)}-1)(2ku_{k-1}'(t)+V(t)u_k(t))dt
  \end{equation}

       \begin{equation}
    \label{eq68a}
    u_k'(x)=\frac{1}{\epsilon}\left(2ku_{k-1}'(x)+V(x)u_k(x)-\int_{\infty}^{x}e^{-2\epsilon(t-x)}(2ku_{k-1}'(t)+V(t)u_k(t))'dt\right)
  \end{equation}
which gives \eqref{eq:estder1} by induction.
\end{proof}

Proposition \ref{jostf} and Lemma \ref{sph} apply to the cases
$|\epsilon|\leqslant  1/\langle
x\rangle$ and $|\epsilon|\geqslant  1/\langle
x\rangle$ respectively, but it is more convenient to have an
expansion for $|\epsilon|\leqslant  1$. This is established in the
following.
\begin{Corollary}\label{f1exp} For $\pm(x-x_{\mp})\geqslant 0$ and $|\epsilon|\leqslant  1$, the exponentially decaying solution $y_{\pm}$ satisfies
\begin{equation}\label{f1expa}
y_{\pm}(x)=r_{\pm}(\epsilon)e^{\mp\epsilon x}\(\Phi_1^{\pm}(x)+\frac{B_2^{\pm}(x)\epsilon^{m-1}\ln\epsilon}{1+\epsilon \langle x \rangle}+R_f^{\pm}(x;\epsilon)\)
\end{equation}
where $r_{\pm}(\epsilon)$ are of the form \eqref{defr}, the modified Bessel functions $\tilde{\Phi}_1^\pm(x)\sim 1$ and $B_{2}^{\pm}(x)\sim const.x$ for $x\to \pm\infty$, and \eqref{fhat} holds with $R_f^{\pm}(x;\epsilon)$ replacing $\hat{f}_a$. Also

\begin{equation}\label{f1d}
y_{\pm}'(x)=r_{\pm}(\epsilon)e^{\mp\epsilon x}\(\Phi_1'^{\pm}(x)+\frac{B_2'{\pm}(x)\epsilon^{m-1}\ln\epsilon}{1+\epsilon \langle x \rangle}+r_f^{\pm}(x;\epsilon)\)
\end{equation}
where \eqref{fhat} holds with $\langle x\rangle r_f'^{\pm}(x)$ replacing $\hat{f}_a$.
\end{Corollary}
\begin{proof}
Without loss of generality we show the result for $y_+$. By Proposition \ref{jostf} we see that \eqref{f1expa} holds for $|\epsilon|\leqslant 1/x$, since multiplying by $e^{\epsilon x}$ and $1/(1+\epsilon \langle x \rangle)$ does not change the structure \eqref{cancel0}. For $1/x \leqslant|\epsilon|\leqslant 1$, we only need to show that \eqref{fhat} is valid with $\hat{f}_{a}$ replaced by $R_f(x;\epsilon)$. By \eqref{eq:estder1} we have for $1/x \leqslant|\epsilon|\leqslant 1$
\begin{equation}
  \label{eq79}
  \left|\frac{\partial^n s_+(x;\epsilon)}{\partial \epsilon^n} \right|\lesssim  x^{n+2-m}; \ 0\leqslant  n\leqslant  m+1
\end{equation}
Now $(r_+(\epsilon))^{-1}(1+s_+(x;\epsilon))$ satisfies the estimates for $\hat{f}_{a}$ in \eqref{fhat} by \eqref{eq79} (cf. Note \ref{otherj}). Since $\ds\frac{B_2^+(x)\epsilon^{m-1}\ln \epsilon}{1+\epsilon \langle x \rangle}$ obviously satisfies the estimates in \eqref{fhat}, we see that $R_f$ satisfies \eqref{fhat}. Differentiability of \eqref{f1expa} follows from Lemma \ref{coro} and Lemma \ref{sph}.
\end{proof}

\section{Singularity properties of $\hat{\psi}$ for $x\in \mathbb{R}$}\label{prothm}
Corollary \ref{f1exp} gives the needed properties of $y_{\pm}$
for $\pm(x-x_{\mp})\geqslant 0$. Now we extend the results to all
$x\in \mathbb{R}$.
\subsection{Exponentially decaying functions $y_\pm$ on the real line}
Assume $V(x)=v_1x^{-m}$ for $x\geqslant x_+$ and $V(x)=v_2x^{-m}$ for $x\leqslant  x_-$. We now show the following

\begin{Proposition}
\label{allx}
The exponentially decaying functions $y_{\pm}$ are analytic in $\epsilon\in\mathbb{H}$ for all $x\in\mathbb{R}$. Moreover, for $|\epsilon|\leqslant  1$ they satisfy
\begin{equation}\label{y1expa}
y_{\pm}(x)=r_{\pm}(\epsilon)e^{\mp\epsilon x}\(r_0^{\pm}(x)+\frac{r_1^{\pm}(x)\epsilon^{m-1}\ln\epsilon}{1+\epsilon\langle x\rangle}+R_0^{\pm}(x;\epsilon)\)
\end{equation}
where $|r_{0}^{\pm}(x)|\lesssim \langle x\rangle$ and $|r_{1}^{\pm}(x)|\lesssim\langle x\rangle^2$ for all $x\in\mathbb{R}$, and for ${\pm}x\geqslant {\pm}x_{\pm}$, $r_{0}^{\pm}(x)$ and $r_{1}^{\pm}(x)$ are explicit modified Bessel functions with $|r_{0}^{\pm}(x)|\lesssim 1, |r_{1}^{\pm}(x)|\lesssim\langle x\rangle$ (cf. Corollary \ref{f1exp}). Furthermore
 $R_0^{\pm}(x;0)=0$ and for arbitrarily small $\delta>0$ we have
\begin{multline}
\label{r0bd} \sup_{\substack{0\leqslant  k\leqslant  m-1 \\ \pm(x-x_{\mp}) \geqslant 0}}\left|\langle x\rangle^{-k}\frac{\partial^k R_0^{\pm}(x;\epsilon)}{\partial \epsilon^k}\right|\lesssim 1
 ; \ \sup_{ \substack{m\leqslant  k\leqslant  m+1\\\pm(x-x_{\mp}) \geqslant 0}}\left||\epsilon|^{k-m+\delta}\langle x\rangle^{-m}\frac{\partial^k R_0^{\pm}(x;\epsilon)}{\partial \epsilon^k}\right|\lesssim 1
 \\
 \sup_{\substack{0\leqslant  k\leqslant  m-1\\\pm(x-x_{\mp}) \leqslant 0}}\left|\langle x\rangle^{-k-1}\frac{\partial^k R_0^{\pm}(x;\epsilon)}{\partial \epsilon^k}\right|\lesssim 1
 ; \ \sup_{\substack{m\leqslant  k\leqslant  m+1\\\pm(x-x_{\mp}) \leqslant 0}}\left||\epsilon|^{k-m+\delta}\langle x\rangle^{-m-1}\frac{\partial^k R_0^{\pm}(x;\epsilon)}{\partial \epsilon^k}\right|\lesssim
 1
 \end{multline}

In addition, \eqref{y1expa} is differentiable in $x$, i.e. \eqref{r0bd} holds with $R_0^{\pm}(x;\epsilon)$ replaced with $\langle x\rangle R_0'^{\pm}(x;\epsilon)$.

 For $|\epsilon|\geqslant 1$ we have
 \begin{equation}\label{y1expb}
y_{\pm}(x;\epsilon)=e^{\mp\epsilon x}(1+s_{\pm}(x;\epsilon))
\end{equation}
where $s_{\pm}(x;\epsilon)$ satisfies \eqref{eq:estder1} for $\pm(x-x_{\mp}) \geqslant 0$, and
\begin{equation}
  \label{eq:estder1a}
  \left|\frac{\partial^n s_{\pm}(x;\epsilon)}{\partial \epsilon^n} \right|\lesssim  |\epsilon|^{-1}\langle x\rangle^n;\ \left|\frac{\partial^n s_{\pm}'(x;\epsilon)}{\partial \epsilon^n} \right|\lesssim  |\epsilon|^{-1}\langle x\rangle^n; \ 0\leqslant  n\leqslant  m+1
\end{equation} for $\pm(x-x_{\mp}) \leqslant 0$.

\end{Proposition}
\begin{proof} We only analyze $y_+$ since $y_-$ is similar.
For $x\geqslant x_+$, the behavior of $y_+$ for small $\epsilon$ is given in Corollary \ref{f1exp} and the behavior for large $\epsilon$ is given in Lemma \ref{sph}.

First assume $|\epsilon|\leqslant  1$. In the middle region $x_-\leqslant  x\leqslant  x_+$ there exist two linearly independent solutions analytic in $\epsilon$ (by analytic dependence on parameters of ODE) and the $x$ bounds are obvious. Clearly $y_+$ is a linear combination of these two solutions, and thus the analytic structure of $y_+$ is preserved (cf. Corollary \ref{f1exp}). For $x\leqslant  x_-$, by \eqref{eq:eqests} we can assume $x_-$ is large enough such that $y_-(x_-)\ne 0$. By standard ODE results $y_+$ satisfies
\begin{equation}\label{ode2}
y_+(x;\epsilon)=y_-(x;\epsilon)\(\frac{y_+(x_-;\epsilon)}{y_-(x_-;\epsilon)}-W(\epsilon)\int_{x_-}^x \frac{1}{y_-^2(t;\epsilon)}dt\)
\end{equation}
where $W(\epsilon)=y_+y_-'-y_+'y_-$ is the Wronskian.

One can verify by direct calculation that $y_+$ solves \eqref{eq:hom1} for $x\leqslant  x_-$ and is differentiable at $x_-$.
By Corollary \ref{f1exp} and the reasoning above we see that for $x_- \leqslant x\leqslant x_+$
\begin{equation}\label{y1mid}
y_+(x;\epsilon)=r(\epsilon)e^{-\epsilon x}\(\Phi_0(x)+\frac{\hat{B}_2(x)\epsilon^{m-1}\ln\epsilon}{1+\epsilon \langle x \rangle}+\hat{R}_f(x;\epsilon)\)
\end{equation}
where $\Phi_0$ and $\hat{B}_2$ are smooth, and with \eqref{fhat} holds with $\hat{R}_f$ replacing $\hat{f}_a$.
\begin{Note}\label{wron}
Assume $|\epsilon|\leqslant 1$. Since the Wronskian $W$ is independent of $x$, by direct calculation using \eqref{y1mid} and \eqref{f1expa} at $x_-$ for $|\epsilon|\leqslant 1$ we have \begin{equation}\label{wron1}
W(\epsilon)=r_+(\epsilon)r_-(\epsilon)(q_1\epsilon^{m-1}\ln\epsilon+q_2(\epsilon))
\end{equation}
where $q_1$ is a constant, and by Corollary \ref{f1exp} we have $\ds \frac{d^k q_2(\epsilon)}{d\epsilon^k}$ is bounded for $0\leqslant k\leqslant m-1$, and $\ds \epsilon^{k-m+\delta}\frac{d^k q_2(\epsilon)}{d\epsilon^k}$ is bounded for $m\leqslant k\leqslant m+1$ and arbitrarily small $\delta>0$. In the absence of zero-energy resonance, $y_+$ and $y_-$ are linearly independent for small $\epsilon$, and thus $q_2(0)=W(0)\ne 0$.

Assume $|\epsilon|\geqslant 1$. By direct calculation using \eqref{y1expb} at $x_-$ for $|\epsilon|\geqslant 1$, denoting
\begin{equation}\label{eqq3}q_3(\epsilon)=W(\epsilon)-2\epsilon\end{equation}
we have $\ds \frac{d^k q_3(\epsilon)}{d\epsilon^k}$ is bounded for $0\leqslant  k\leqslant  m+1$.

\end{Note}

By \eqref{ode2} we have
$$e^{\epsilon x}y_+(x;\epsilon)=r(\epsilon)\(\hat{\Phi}_1(x)+\frac{B_3(x)\epsilon^{m-1}\ln\epsilon}{1+\epsilon\langle x \rangle}+x\hat{f}_b(x;\epsilon)\)$$
where
$\ds \hat{\Phi}_1(x)=\Phi_1^-(x)\(\frac{\Phi_1^+(x_-)}{\Phi_1^-(x_-)}-W(0)\int_{x_-}^x\frac{1}{{\Phi_1^-}^2(t)}dt\)$ satisfies $|\hat{\Phi}_1(x)|\lesssim \langle x\rangle$, $B_3$ is a smooth function with $|B_3|\lesssim \langle x\rangle^2$
and \eqref{fhat} holds with $\hat{f}_a$ replaced by $\hat{f}_b$ by \eqref{ode2}, which  can be checked using the fact that   for any bounded function $f$ we have
$$ \left|e^{2\epsilon x}\int_{x_-}^x e^{-2\epsilon t} t^k f(t)dt\right|\lesssim |x|^{k+1}$$
Thus \eqref{y1expa} is valid for $x\leqslant x_-$.

Assume $|\epsilon|\geqslant 1$. In the middle region $x_-\leqslant   x\leqslant  x_+$ the result \eqref{y1expb} follows from Lemma \ref{sph}.

For $x\leqslant x_-$, we denote $f_{\pm}=1+s_{\pm}$ and by \eqref{ode2} and Lemma \ref{sph} we have
\begin{equation}\label{s2s3}
f_+(x;\epsilon)=e^{\epsilon x}y_+(x;\epsilon)=f_-(x;\epsilon)(1+s_3(x;\epsilon))\end{equation}
where
\begin{multline}\label{s3part}
s_3(x;\epsilon)=-1+\frac{f_+(x_-;\epsilon)}{f_-(x_-;\epsilon)}e^{2\epsilon (x-x_-)}-(2\epsilon +q_3(\epsilon))\int_{x_-}^x \frac{e^{2\epsilon (x-t)}}{f_-^2(t;\epsilon)}dt
\\=\frac{s_+(x_-;\epsilon)-s_-(x_-;\epsilon)}{f_-(x_-;\epsilon)}e^{2\epsilon (x-x_-)}-q_3(\epsilon)\int_{x_-}^x \frac{e^{2\epsilon (x-t)}}{f_-^2(t;\epsilon)}dt+2\epsilon \int_{x_-}^x \frac{e^{2\epsilon (x-t)}s_-(t;\epsilon)(s_-(t;\epsilon)+2)}{f_-^2(t;\epsilon)}dt\end{multline}
Using \eqref{eq:estder1} in Lemma \ref{sph} we see that the first term in \eqref{s3part} satisfies the estimates in \eqref{eq:estder1a}. By integration by parts we obtain
\begin{equation}q_3(\epsilon)\int_{x_-}^x \frac{e^{2\epsilon (x-t)}}{f_-^2(t;\epsilon)}dt=\frac{q_3(\epsilon)}{2\epsilon}
\bigg(\frac{e^{2\epsilon (x-x_-)}}{f_-^2(x_-;\epsilon)}
-\frac{1}{f_-^2(x;\epsilon)}+\int_{x_-}^x e^{2\epsilon (x-t)}\(\frac{1}{f_-^2(t;\epsilon)}\)'dt\bigg)
\end{equation}
and
\begin{multline}2\epsilon \int_{x_-}^x \frac{e^{2\epsilon (x-t)}s_-(t;\epsilon)(s_-(t;\epsilon)+2)}{f_-^2(t;\epsilon)}dt=
\frac{e^{2\epsilon (x-x_-)}s_-(x_-;\epsilon)(s_-(x_-;\epsilon)+2)}{f_-^2(x_-;\epsilon)}\\
-\frac{s_-(x;\epsilon)(s_-(x;\epsilon)+2)}{f_-^2(x_;\epsilon)}
+\int_{x_-}^x e^{2\epsilon (x-t)}\(\frac{s_-(t;\epsilon)(s_-(t;\epsilon)+2)}{f_-^2(t;\epsilon)}\)'dt\end{multline}
Thus using the estimates in Lemma \ref{sph} as well as \eqref{s3part} we see that the first estimate in \eqref{eq:estder1a} holds with $s_+$ replaced by $s_3$.

To show \eqref{eq:estder1a} for $s_+$, note that by \eqref{s2s3} we have $s_+(x;\epsilon)=s_-(x;\epsilon)+s_3(x;\epsilon)-s_-(x;\epsilon)s_3(x;\epsilon)$. Since $s_-$ satisfies \eqref{eq:estder1} and $s_3$ satisfies $\left|\frac{\partial^n s_{3}(x;\epsilon)}{\partial \epsilon^n} \right|\lesssim  |\epsilon|^{-1}\langle x\rangle^n$,
 we see that the first estimate in \eqref{eq:estder1a} holds for $s_+$.

To estimate $s_+'$, note that by \eqref{ode2} and \eqref{s2s3}
\begin{multline}s_+'(x;\epsilon)=(s_-'(x;\epsilon)+2\epsilon+2\epsilon s_-(x;\epsilon))\(1+s_3(x;\epsilon)\)-\frac{W(\epsilon)}{h_-(x;\epsilon)}\\
=2\epsilon(s_-(x;\epsilon)+s_3(x;\epsilon)-s_+(x;\epsilon))+O(\epsilon^{-1})
=2\epsilon s_-(x;\epsilon)s_3(x;\epsilon)+O(\epsilon^{-1})
\end{multline}
Thus the second inequality of \eqref{eq:estder1a} holds.

The proof for $s_-$ is similar by symmetry.
\end{proof}

\subsection{Analysis of $\hat{\psi}$}\label{hat}
Recall that the solution of \eqref{psiwave} is given in \eqref{invlap}.
Note that $W(\epsilon)\ne 0$ for $\epsilon\in \overline{\mathbb{H}}$ by the assumption of no bound state and no zero energy resonance.

To obtain the estimates for $\frac{\sin(t\sqrt{A})}{\sqrt{A}}\psi_1$ and $\cos(t\sqrt{A})\psi_0$ in Theorem
\ref{xbounds}, we need the $\epsilon$-expansions of $\mathcal{G}(\psi_1)$ and $\mathcal{G}(\epsilon\psi_0)$:
\begin{Proposition}
\label{hard1}
For $\epsilon\in \overline{\mathbb{H}}\backslash \{0\}$ we have
\begin{equation}\label{hardeq}
\mathcal{G}(\psi_1)=r_3(x)\frac{\epsilon^{m-1}\ln\epsilon}{(1+\epsilon\langle x\rangle)^{m+2}}+\mathcal{G}_0(\psi_1(x))+R(x;\epsilon)
\end{equation}
where
\begin{equation}\label{r3est}
||\langle x\rangle^{-2}r_3(x)||_{\infty}\lesssim ||\langle x\rangle^2\psi_1(x)||_{1}
\end{equation}

\begin{equation}\label{defg0}\mathcal{G}_0(\psi):=\frac{1}{2(1+\epsilon)}\(\int_{\infty}^{x}e^{-\epsilon (u-x)}\psi(u)du-\int_{-\infty}^{x}e^{\epsilon (u-x)}\psi(u)du\)\end{equation}

\begin{multline}\label{rrbd}
||\langle x\rangle^{-k-2}\frac{\partial^{k} }{\partial \epsilon^{k}}R(x,\epsilon)||_{\infty}\lesssim (1+|\epsilon|)^{-2}||\langle x\rangle^{k+2}\psi_1(x)||_{1} \ \ (0\leqslant  k\leqslant  m-1)
\\
||\langle x\rangle^{-m-2}\frac{\partial^{m} }{\partial \epsilon^{m}}R(x,\epsilon)||_{\infty}\lesssim (|\epsilon|^{\delta}+|\epsilon|)^{-2}||\langle x\rangle^{m+2}\psi_1(x)||_{1}
\end{multline}
for arbitrary $\delta>0$.
\end{Proposition}

\begin{proof}
The fact that $\mathcal{G}(\psi_1)$ is analytic in $\epsilon\in \mathbb{H}$ follows easily from Proposition \ref{allx}.

We first consider the case $ |\epsilon|\leqslant  1$.
By \eqref{y1expa} we have
\begin{equation}\label{numg}
y_-(x)\int_{\infty}^{x}y_+(u)\psi_1(u)du-y_+(x)\int_{-\infty}^{x}y_-(u)\psi_1(u)du=r_+(\epsilon)r_-(\epsilon) \sum_{k,j=1,2}\tilde{G}_{j,k}(x,\epsilon)
\end{equation}
where \begin{multline}\tilde{G}_{j,k}(x,\epsilon)=e^{\epsilon x}G_j^-(x;\epsilon)\int_{\infty}^{x}e^{-\epsilon u}G_k^+(u;\epsilon)\psi_1(u)du -e^{-\epsilon x}G_j^+(x;\epsilon)\int_{-\infty}^{x}e^{\epsilon u}G_k^-(u;\epsilon)\psi_1(u)du
\end{multline}
$\ds G_1^{\pm}(x;\epsilon)=\frac{r_1^{\pm}(x)\epsilon^{m-1}\ln\epsilon}{1+\epsilon \langle x\rangle} $, and  $G_2^{\pm}(x;\epsilon)=r_0^{\pm}(x)+R_0^{\pm}(x;\epsilon)$. By Proposition \ref{allx} clearly \eqref{r0bd} holds with  $G_2^{\pm}$ instead of $R_0^{\pm}$. We denote for $(j,k)=(0,1)$ or $(1,0)$ $$\hat{G}_{j,k}(x)=\(r_j^{-}(x)\int_{\infty}^{x}r_k^+(u)\psi_1(u)du-r_j^{+}(x)\int_{-\infty}^{x}r_k^-(u)\psi_1(u)du \)$$
By direct calculation for $0\leqslant k\leqslant m-1$
\begin{multline}\label{eq98}
\bigg|\frac{\partial^k}{\partial \epsilon^k}\bigg( \tilde{G}_{2,1}(x,\epsilon)-\epsilon^{m-1}\ln\epsilon\hat{G}_{0,1}(x)  \bigg)\bigg|\lesssim \sum_{\substack{i+j+l=k\\i< m-1}}|\epsilon|^{m-1-i-\delta} \langle x\rangle^{1+j} {\big \|}\langle x\rangle^{2+l}\psi_1(x) \big \|_1\\
\lesssim \langle x\rangle^{1+k} \|\langle x\rangle^{2+k}\psi_1(x)\|_1
\end{multline}
\begin{multline}\bigg|\frac{\partial^m}{\partial \epsilon^m}\bigg( \tilde{G}_{2,1}(x,\epsilon)-\epsilon^{m-1}\ln\epsilon\hat{G}_{0,1}(x)  \bigg)\bigg|\lesssim \sum_{\substack{i+j+l=m\\i< m}}|\epsilon|^{m-1-i-\delta} \langle x\rangle^{1+j} {\big \|}\langle x\rangle^{2+l}\psi_1(x) \big \|_1\\
\lesssim |\epsilon|^{-\delta}\langle x\rangle^{1+m} \|\langle x\rangle^{2+m}\psi_1(x)\|_1
\end{multline}
and it can be similarly shown that for $0\leqslant k\leqslant m$
$$\bigg|\frac{\partial^k}{\partial \epsilon^k}\bigg( \tilde{G}_{1,2}(x,\epsilon)-\epsilon^{m-1}\ln\epsilon\hat{G}_{1,0}(x)  \bigg)\bigg|\lesssim (1+|\epsilon|^{m-k-\delta})\langle x\rangle^{2+k} \|\langle x\rangle^{1+k}\psi_1(x)\|_1$$
Also for $0\leqslant k\leqslant m$
\begin{multline}\bigg|\frac{\partial^k}{\partial \epsilon^k} \tilde{G}_{1,1}(x,\epsilon)\bigg|\lesssim \sum_{i+j+l=k}
\frac{|\epsilon|^{2m-2-i}(|\ln\epsilon|+1)^2\langle x\rangle^{2+j}}{1+|\epsilon|\langle x\rangle} \left \|\frac{\langle x\rangle^{2+l}\psi_1(x)}{1+|\epsilon|\langle x\rangle} \right \|_1 \\
\lesssim \langle x\rangle^{1+k} \|\langle x\rangle^{1+k}\psi_1(x)\|_1
\end{multline}
\begin{equation}\label{eq101}
\bigg|\frac{\partial^k}{\partial \epsilon^k} \tilde{G}_{2,2}(x,\epsilon)\bigg|\lesssim \sum_{i+j+l=k}(|\epsilon|^{m-i-\delta}+1)\langle x\rangle^{1+j} \left \|\langle x\rangle^{1+l}\psi_1(x) \right \|_1
\end{equation}
Using \eqref{eq98}-\eqref{eq101} as well as \eqref{wron1} we have
\begin{equation}\label{eqrhat}
\mathcal{G}(\psi_1)=\frac{r_+(\epsilon)r_-(\epsilon)\sum_{k,j=1,2}\tilde{G}_{j,k}(x,\epsilon)}{W(\epsilon)}
= r_3(x)\epsilon^{m-1}\ln\epsilon+\hat{R}(x,\epsilon)\end{equation}
where $r_3(x)=(\hat{G}_{0,1}(x)+\hat{G}_{1,0}(x))/W(0)$ satisfies \eqref{r3est} and \eqref{rrbd} holds with $\hat{R}$ replacing $R$. \eqref{eqrhat} and \eqref{hardeq} give \begin{equation}\label{rrh}
R(x,\epsilon)-\hat{R}(x,\epsilon)=r_3(x)\epsilon^{m-1}\ln\epsilon\(1-\frac{1}{(1+\epsilon\langle x\rangle)^{m+2}}\)-\mathcal{G}_0(\psi_1)
\end{equation}
Now by direct calculation \eqref{rrbd} holds with the right side of \eqref{rrh} replacing $R$, using $$\left|\int_{\pm\infty}^{x}(u-x)^ke^{\mp\epsilon (u-x)}\psi_1(u)du\right| \lesssim \langle x\rangle^k ||\langle x\rangle^{k}\psi_1(x)||_{1} $$
for the estimates involving $\mathcal{G}_0(\psi)$. Thus \eqref{hardeq} is valid for $|\epsilon|\leqslant 1$.

For $|\epsilon|\geqslant 1$, we analyze the numerator of \eqref{invlap} by noting that (recall that $f_{\pm}=1+s_{\pm}$)
\begin{multline}
f_-(u;\epsilon) \int_{\infty}^{x}e^{-\epsilon (u-x)}f_+(u;\epsilon)\psi_1(u)du =
\int_{\infty}^{x}e^{-\epsilon (u-x)}\psi_1(u)du\\ + s_-(x;\epsilon) \int_{\infty}^{x}e^{-\epsilon (u-x)}f_+(u;\epsilon)\psi_1(u)du+\int_{\infty}^{x}e^{-\epsilon (u-x)}s_+(u;\epsilon)\psi_1(u)du
\end{multline}
 and $s_{\pm}$ satisfies \eqref{eq:estder1a}. The same type of expansion works for the integral from $-\infty$ to $x$. Thus by \eqref{invlap} we have
$$\mathcal{G}(\psi_1)=\frac{2\epsilon\mathcal{G}_0(\psi_1)+(1+\epsilon) R_3(x;\epsilon)}{W(\epsilon)}$$ for some $R_3$ satisfying estimates of the type \eqref{rrbd}. Note that $W(\epsilon)$ satisfies \eqref{eqq3}. Obviously \eqref{rrbd} holds with $\ds r_3(x)\frac{\epsilon^{m-1}\ln\epsilon}{(1+\epsilon\langle x\rangle)^{m+2}}$ replacing $R$. Thus \eqref{hardeq} is valid for $|\epsilon|\geqslant 1$ as well.
\end{proof}

Similarly we have
\begin{Proposition}
\label{hard0}
For $\epsilon\in \overline{\mathbb{H}}\backslash \{0\}$ we have
\begin{equation}\label{hardeq0}\mathcal{G}(\epsilon\psi_0)=r_4(x)\frac{\epsilon^{m}\ln\epsilon}{(1+\epsilon\langle x\rangle)^{m+3}}-\frac{ \psi_0(x)\epsilon }{(\epsilon+1)^2} +\frac{\epsilon}{\epsilon+1}\mathcal{G}_0(\psi_0(x))+\frac{\epsilon}{\epsilon+1}\mathcal{G}_1(\psi_0'(x))+ R_4(x;\epsilon)\end{equation}
where
\begin{equation}\label{eqr4}
||\langle x\rangle^{-2}r_4(x)||_{\infty}\lesssim ||\langle x\rangle^2\psi_0(x)||_{1}
\end{equation}
$$ ||\psi_0(x)||_{\infty} \lesssim ||\psi_0'(x)||_{1}$$
$$\mathcal{G}_1(\psi):=\frac{1}{2(1+\epsilon)}\(\int_{\infty}^{x}e^{-\epsilon (u-x)}\psi(u)du+\int_{-\infty}^{x}e^{\epsilon (u-x)}\psi(u)du\)$$
\begin{multline}\label{rrbd4}
||\langle x\rangle^{-k-2}\frac{\partial^{k} }{\partial \epsilon^{k}}R_4(x;\epsilon)||_{\infty}\lesssim (1+|\epsilon|)^{-2}(||\langle x\rangle^{k+2}\psi_0(x)||_{1}+||\langle x\rangle^{k+2}\psi_0'(x)||_{1}) ~ (0\leqslant  k\leqslant  m)\\
||\langle x\rangle^{-m-3}\frac{\partial^{m+1} }{\partial \epsilon^{m+1}}R_4(x;\epsilon)||_{\infty}\lesssim (|\epsilon|^{\delta}+|\epsilon|)^{-2}(||\langle x\rangle^{m+3}\psi_0(x)||_{1}+||\langle x\rangle^{m+3}\psi_0'(x)||_{1})
\end{multline}
\end{Proposition}

\begin{proof}
The proof is essentially the same as that of Proposition \ref{hard1}. In the case $|\epsilon |\leqslant 1$, we simply use \eqref{numg}-\eqref{eq101} with $\psi_1$ replaced by $\epsilon\psi_0$ for $0\leqslant k\leqslant m+1$, which gives the counterpart of \eqref{eqrhat} as
\begin{equation}\label{eqrhat2}
\mathcal{G}(\epsilon\psi_0)= r_4(x)\epsilon^{m}\ln\epsilon+\hat{R}_4(x,\epsilon)\end{equation}
where $r_4$ satisfies \eqref{eqr4} and $\hat{R}_4$ satisfies the same estimates as $R_4$ (see
\eqref{rrbd4}). The rest follows in the same way as \eqref{rrh}. Note that we have the obvious inequality $|\psi_0(x)|\lesssim ||\psi_0'(x)||_{1} $.

For $|\epsilon |\geqslant 1$ we use integration by parts to get
\begin{multline}\label{intp1}
\epsilon e^{\epsilon x} \int_{\infty}^{x}y_+(u)\psi_0(u)du=\epsilon e^{\epsilon x}\int_{\infty}^{x}e^{-(\epsilon+1) u}e^u(1+s_+(u;\epsilon))\psi_0(u)du \\
=-\frac{\epsilon }{\epsilon+1}(1+s_+(x;\epsilon))\psi_0(x)+\frac{\epsilon}{\epsilon+1}e^{\epsilon x}\int_{\infty}^{x}e^{-\epsilon u}\left(\psi_0'(u)+\psi_0(u)\right)du  \\+\frac{\epsilon}{\epsilon+1}e^{\epsilon x}\int_{\infty}^{x}e^{-\epsilon u}\left(s_+'(u;\epsilon)\psi_0(u)+s_+(u;\epsilon)(\psi_0'(u)+\psi_0(u))\right)du
\end{multline}
Similarly
\begin{multline}\label{intp2}
-\epsilon e^{-\epsilon x}\int_{-\infty}^{x}y_-(u)\psi_0(u)du
=-\frac{\epsilon }{\epsilon+1}(1+s_-(x;\epsilon))\psi_0(x)+\frac{\epsilon}{\epsilon+1}e^{-\epsilon x}\int_{-\infty}^{x}e^{\epsilon u}\left(\psi_0'(u)-\psi_0(u)\right)du  \\+\frac{\epsilon}{\epsilon+1}e^{-\epsilon x}\int_{-\infty}^{x}e^{\epsilon u}\left(s_-'(u;\epsilon)\psi_0(u)+s_-(u;\epsilon)(\psi_0'(u)-\psi_0(u))\right)du
\end{multline}
Since $s_{\pm}$ satisfy \eqref{eq:estder1a}, by \eqref{invlap}, \eqref{intp1} and \eqref{intp2} we have
$$\mathcal{G}(\epsilon\psi_0)=-\frac{2\epsilon \psi_0(x) }{(\epsilon+1)W(\epsilon)}+\frac{2\epsilon\mathcal{G}_0(\psi_0(x))+\mathcal{G}_1(\psi_0'(x))}{W(\epsilon)}+
\frac{\epsilon \tilde{R}_3(x;\epsilon)}{W(\epsilon)}$$
for some $\tilde{R}_3$ satisfying estimates of the type \eqref{rrbd4}. The rest follows from \eqref{eqq3}.
\end{proof}

\section{Time decay of $\psi$}\label{timedec}

\begin{proof}[Proof of Theorem \ref{xbounds}] We focus on the case $\frac{\sin(t\sqrt{A})}{\sqrt{A}}\psi_1$ since the proof for $\cos(t\sqrt{A})\psi_0$  is analogous.

We have using the decomposition \eqref{hardeq} that
\begin{multline}\label{finalpre}
\psi(x,t)=\mathcal{L}^{-1}\hat{\psi}(x,\epsilon)\\=\frac{1}{2\pi i}\int_{-\infty i}^{\infty i}e^{\epsilon t}\left(r_3(x)\frac{\epsilon^{m-1}\ln \epsilon}{(1+\epsilon\langle x\rangle)^{m+1}}+R(x;\epsilon)\right)d\epsilon+\mathcal{L}^{-1}\mathcal{G}_0(\psi_1(x))\\
\end{multline}
where the inverse Laplace transform $\mathcal{L}^{-1}$ can be represented using the Bromwich integral formula for the terms $\ds r_3(x)\frac{\epsilon^{m-1}\ln \epsilon}{(1+\epsilon\langle x\rangle)^{m+1}}$ and $R$, and we have
\begin{Lemma} We have the estimate
\begin{equation}\label{g0est}
|\mathcal{L}^{-1}\mathcal{G}_0(\psi_1(x))|\lesssim (\langle x\rangle/\langle t\rangle)^{m+1}\|\langle x\rangle^{m+1}\psi_1(x)\|_1
\end{equation}
\end{Lemma}
\begin{proof}
By direct calculation
$$\mathcal{G}_0(\psi_1(x))=\mathcal{L}\(\frac12\int_{t+x}^{x}e^{-t+u-x}\psi_1(u)du- \frac12\int_{x-t}^{x}e^{-t-u+x}\psi_1(u)du\)$$
The estimate \eqref{g0est} follows from integration by parts. For instance we have
\begin{multline}
\int_{t+x}^{x}e^{-t+u-x}\psi_1(u)du
=\frac{e^{-t+u-x}}{(\sqrt{u^2+1})^{m+1}}\bigg|_{t+x}^x\int_{t+x}^x\(\sqrt{u^2+1}\)^{m+1}\psi_1(u)du\\- \int_{t+x}^{x}\(\frac{e^{-t+u-x}}{\(\sqrt{u^2+1}\)^{m+1}}\)'\(\int_{0}^u \(\sqrt{v^2+1}\)^{m+1}\psi_1(v)dv \)du
\end{multline}
where the first term on the right side can be estimated using the elementary inequality
$$\frac{1}{(\sqrt{(t+x)^2+1})^{m+1}}\lesssim (\langle x \rangle/\langle t\rangle)^{m+1}$$
and the last term can be estimated by noting that
\begin{multline}
 \left|\int_{t+x}^{x}\(\frac{e^{-t+u-x}}{(\sqrt{u^2+1})^{m+1}}\)'du\right|\leqslant \left|\int_{t+x}^{x}\frac{e^{-t+u-x}}{(\sqrt{u^2+1})^{m+1}}du\right| \leqslant  \int_{x}^{x+t/2}\frac{e^{-t+u-x}}{(\sqrt{u^2+1})^{m+1}}du \\
+\int_{x+t/2}^{x+t}\frac{e^{-t+u-x}}{(\sqrt{u^2+1})^{m+1}}du
\lesssim e^{-t/2}+\sup_{t/2\leqslant  v\leqslant  t}\frac{1}{(\sqrt{(x+v)^2+1})^{m+1}}\lesssim (\langle x \rangle/\langle t\rangle)^{m+1}
\end{multline}
\end{proof}

The leading term on the right hand side of \eqref{finalpre} can be estimated by contour deformation. Since there is a branch cut along $\mathbb{R}^-$, the original integration path from $-i\infty$ to $i\infty$ can be deformed into the path consisting of the vertical line from $-1/2-i\infty$ to $-1/2$, the horizonal line from $-1/2$ to $0$ on the lower side of the branch cut, the horizonal line from $0$ to $-1/2$ on the upper side of the branch cut, and the vertical line from $-1/2$ to $-1/2+i\infty$.
Therefore, by Watson's Lemma we have for large $t$
\begin{multline}\label{wats}
\frac{1}{2\pi i}\int_{-\infty i}^{\infty i}e^{\epsilon t}r_3(x)\frac{\epsilon^{m-1}\ln \epsilon}{(1+\epsilon\langle x\rangle)^{m+1}}d\epsilon =-\int_{0}^{-1/2}e^{\epsilon t}\frac{r_3(x)\epsilon^{m-1}}{(1+\epsilon\langle x\rangle)^{m+1}}d\epsilon+O(e^{-t/2})r_3(x)\\
=(-1)^{m+1}(m-1)!r_3(x)t^{-m}(1+O(t^{-1})\langle x\rangle)+O(e^{-t/2})r_3(x)
\end{multline}
Note that the left side of \eqref{wats} is obviously bounded by $const.|r_3(x)|$ uniformly for all $t\geqslant0$ and $r_3$ satisfies \eqref{r3est}.
Finally by integration by parts and \eqref{rrbd} we have for $t>0$
$$\left|\frac{1}{2\pi i}\int_{-\infty i}^{\infty i}e^{\epsilon t}R(x;\epsilon)d\epsilon\right|\lesssim t^{-m}\left|\int_{-\infty i}^{\infty i}e^{ \epsilon t}\frac{\partial^m R(x;\epsilon)}{\partial \epsilon^m}d\epsilon \right|$$
Since $R(x;\epsilon)$ and $\ds \frac{\partial^m R(x;\epsilon)}{\partial \epsilon^m}$ satisfy \eqref{rrbd}, we have for $t\geqslant 0$
\begin{equation}\label{rl1est}
||\langle x\rangle^{-m-2}\mathcal{L}^{-1}R(x;\epsilon)||_{\infty}\lesssim \langle t\rangle ^{-m}||\langle x\rangle^{m+2}\psi_1(x)||_{1}
\end{equation}
Furthermore, it follows from the Riemann-Lebesgue lemma that
\begin{equation}\label{rl1est1}
\lim_{t\to\infty}t^m\mathcal{L}^{-1}R(x;\epsilon)=0
\end{equation}
We define $\hat{r}_1(x)=(-1)^{m+1}(m-1)!r_3(x)$ and
$$R_1(x,t)=\langle t\rangle^{m}\mathcal{L}^{-1}\(\mathcal{G}_0(\psi_1(x))+R(x;\epsilon)+r_3(x)\frac{\epsilon^{m-1}\ln \epsilon}{(1+\epsilon\langle x\rangle)^{m+1}}\)-\hat{r}_1(x)$$
The result for $\frac{\sin(t\sqrt{A})}{\sqrt{A}}\psi_1$ in Theorem \ref{T1} then follows from \eqref{finalpre} using \eqref{g0est}, \eqref{wats}, \eqref{rl1est}, and \eqref{rl1est1}; for $\cos(t\sqrt{A})\psi_0$ the estimates follow from Proposition \ref{hard0} in a similar way.
\end{proof}

\subsection{Genericity of decay rate}
\begin{Remark}\label{gende}
In view of the definition of $r_3$ (below \eqref{eqrhat}) it is clear that $r_3$ is nonzero for generic initial condition $\psi_1$, meaning the time decay $t^{-m}$ is generic.
\end{Remark}

\section{More general potentials}\label{extend}
\subsection{Sums of inverse powers}
 Assume  $\ds V(x)=const.x^{-\alpha_1^{\pm}}(1+\sum_{k=1}^n a_k^{\pm} x^{-\beta_k^{\pm}})$ for large ${\pm}x$ where $\alpha_1^{\pm}>2$ and  $\beta_k^{\pm}>0$.

Without loss of generality we study large $x$. Now \eqref{eq:eq7} has the form
\begin{multline}
  \label{eq:eq7n}
  F({\tau})=const.(\epsilon\tau)^{\alpha_1^{+}-1}(1+\sum_{k=1}^n b_k (\epsilon\tau)^{\beta_k^{+}})\\+const.\epsilon^{\alpha_1^{+}-2}\int_0^\tau
   (\tau-u)^{\alpha_1^{+}-1}(1+\sum_{k=1}^n b_k (\epsilon(\tau-u))^{\beta_k^{+}})\frac{F({u})}{{u}({u}+2)}du
\end{multline}
It can be shown that for large $\tau$
\begin{multline}
\frac{F({\tau})}{\epsilon^2\tau(\tau+2)}
=const.(\epsilon\tau)^{\alpha_1^{+}-3}(1+\sum_{k=1}^n b_k  (\epsilon\tau)^{\beta_k^{+}})\left(1+\sum_{k=0}^n \sum_{l=1}^{\infty} \sum_{m=0}^{l}c_{klm}  (\epsilon\tau)^{l(\alpha_1^{+}-2+\beta_k^{+})}(\ln \tau)^m\right)
\end{multline}
where $\beta_0=0$ and $\ln$ terms are only present for $\alpha_1^{+}\in \mathbb{N}$.

The counterpart of \eqref{sm4} is now the expansion
$$s(x;\epsilon)=\epsilon^{\alpha_1^{+}-2}\sum_{k=0}^n (\tilde{c}_k(x)+\epsilon c_k(x))  \epsilon^{\beta_k^{+}}+\epsilon^{\alpha_1^{+}-2}\ln \epsilon\sum_{k=0}^n (\tilde{C}_k(x)+\epsilon C_k(x))\epsilon^{\beta_k^{+}}+...$$
where terms with higher orders of $\ln$ are omitted and the $\ln$ terms are only present for $\alpha_1^{+}\in \mathbb{N}$.

This implies the counterpart of the expansion \eqref{f1expa}
$$y_+(x;\epsilon)=r(\epsilon)e^{-\epsilon x}(\hat{D}_1(x)\epsilon^{\alpha_1^{+}-1}+\hat{D}_2(x)\epsilon^{\alpha_1^{+}-1}\ln \epsilon+R(x;\epsilon))$$
where $\ds\frac{\partial^k R(x;\epsilon)}{\partial \epsilon^k}$ are bounded by $\epsilon^{\delta-1}$ for $0\leqslant  k\leqslant  \lceil \alpha_1^{+} \rceil$ and some $\delta>0$, and $D_2=0$ except for $\alpha_1\in \mathbb{N}$.

Arguments similar to those showing Proposition \ref{hard1} lead to the same type of expansion for $\hat{\psi}$, which then implies that
$$\frac{\sin(t\sqrt{A})}{\sqrt{A}}\psi_1\sim \hat{r}_1(x)t^{-\alpha}$$
$$\cos(t\sqrt{A})\psi_0\sim\hat{r}_0(x)t^{-\alpha-1}$$
The detailed estimates for the higher order remainder as in Theorem \ref{xbounds} can be obtained in similar ways.

\subsection{Inverse power with higher order correction}
Here we discuss the general $m$ analog of potentials of the type in \cite{sch1}.
Assume $V(x)=V_0(x)+V_1(x)$ and $V_0(x)=c_v/x^m$ for $|x|\geqslant x_+$, $V_1$ is piecewise continuous, $|V_1^{(k)}(x)|\lesssim \langle x \rangle^{-m-k-\delta_1}$ where $0\leqslant k\leqslant m+2$ and $\delta_1>3$. One can show that for $x\geqslant x_+$ and $|\epsilon|\leqslant  1/x$, $\tilde{y}_1$ has the expansion
\begin{equation}
\label{canceln}
\tilde{y}_1(x;\epsilon)=r(\epsilon) \bigg(
\tilde{B}_0(x)+\tilde{B}_1(x)\epsilon^{m-1}\ln\epsilon+f_d(x,\epsilon)\bigg)
\end{equation}
 where $\tilde{B}_k$ solves $f''(x)=V(x)f(x)$ with  $\tilde{B}_0\sim \Phi_1=:B_0$, $\tilde{B}_1\sim B_1$ for large $x$ and $f_d(x,\epsilon)$ satisfies \eqref{fhat}.

Indeed, by \eqref{cancel0} and standard ODE analysis one can find $\tilde{B}_k(x)$ with $\tilde{B}_k''(x)=V(x)\tilde{B}_k(x)$, $|\tilde{B}_k(x)-B_k(x)|\lesssim \langle x\rangle^{2-m-\delta_1}\langle B_k(x)\rangle$ and $|\tilde{B}_k'(x)-B_k'(x)|\lesssim \langle x\rangle^{1-m-\delta_1}\langle B_k'(x)\rangle$.

By \eqref{cancel0} we have
\begin{multline}
\label{cancelnn}
f_+(x;\epsilon)=r(\epsilon) \bigg(
\hat{\Phi}_1(x)+B_1(x)\epsilon^{m-1}\ln\epsilon+f_c(x,\epsilon)\bigg) =:r(\epsilon) \bigg(\phi (x;\epsilon;B) +f_c(x,\epsilon)\bigg)
\end{multline}
where $f_c$ satisfies \eqref{fhat} and
$$f_c''(x,\epsilon)-2\epsilon f_c'(x,\epsilon)-V_0(x)f_c(x,\epsilon)=2\epsilon\phi (x;\epsilon;B')$$

Similarly we write $\tilde{f}_1(x;\epsilon)=r(\epsilon) \(\phi (x;\epsilon;\tilde{B}) +f_c(x,\epsilon)+g_1(x,\epsilon)\)$, which implies $\tilde{g}_1(x,\epsilon)=e^{-\epsilon x}g_1(x,\epsilon)$ satisfies the equation
$$\tilde{g}_1''(x,\epsilon)-\epsilon^2 \tilde{g}_1(x,\epsilon)-V_0(x) \tilde{g}_1(x,\epsilon)=e^{-\epsilon x}\phi_1(x,\epsilon)  +V_1(x) \tilde{g}_1(x,\epsilon)$$
where $\phi_1(x,\epsilon) =2\epsilon \phi (x;\epsilon;\tilde{B}'-B')+\epsilon V_1(x)\hat{f}_c(x,\epsilon)$ and $\hat{f}_c=f_c/\epsilon$. Equivalently we have the integral equation
\begin{equation}\label{ggg}
g_1(x)=\mathcal{G}(e^{-\epsilon x}\phi_1(x;\epsilon))+\mathcal{G}(V_1(x)g_1(x))
\end{equation}
Using Proposition \ref{allx} we see that $|\phi_1(x;\epsilon)|\lesssim |\epsilon|\langle x\rangle^{1-m-\delta_1}$ for all $x\geqslant x_+$ and thus \eqref{ggg} is contractive under the norm $||\langle x\rangle^{m-2+\delta_1}g_1(x)||_{\infty}$, and by taking derivatives of \eqref{ggg} we have
$$\left| \frac{\partial^k g_1(x;\epsilon)}{\partial \epsilon^k} \right| \lesssim x^{2-m-\delta_1+k} \
(0\leqslant  k\leqslant  m-1); \ \ \left| \frac{\partial^k g_1(x;\epsilon)}{\partial \epsilon^k} \right| \lesssim |\epsilon|^{m-k-\delta}x^{2-\delta_1} \
(m\leqslant  k\leqslant  m+1)$$
since
$$\left|e^{\epsilon x}\int_{\infty}^x e^{-2\epsilon t}t^k\phi_1(t;\epsilon)dt\right|\lesssim 1$$
if $k-m-\delta_1<-2$.
Thus if $\delta_1>3$ then the counterpart of Proposition \ref{jostf} holds. The rest of the proof is similar.

\section{Appendix}
For completeness, in this section we provide a short self-contained proof
justifying the use of the Laplace transform.
\begin{Proposition}\label{lapt}
  Assume the initial conditions $f(x)=u(x,0)$ and $g(x)=u_t(x,0)$ are in
  $L^1(\RR)$ and $V\in L^{\infty}(\RR)$. Then, if $\nu>\sqrt 2\|V\|_{\infty}^{\frac12}$  we have
  $\sup_{t>0}e^{-\nu t}\|u(t,\cdot)\|_1<\infty$, and thus $u(t,x)$ is Laplace transform in $t$.
\end{Proposition}
\begin{proof}
  We use the Duhamel principle to write \eqref{wave} in the form
  \begin{multline}
    \label{eq:duhamel}
    u=\mathcal{A}u;\ \mathcal{A}u:=
\frac{f(x-t)+f(x+t)}{2}\\ +\tfrac12\int_{-\infty}^\infty \chi_t(y-x)g(y)dy+\frac{1}{2}\int_0^t\int_{-\infty}^\infty u(y,s)V(y)\chi_{t-s}(y-x))dy ds
  \end{multline}
where $\chi_{a}$ is the characteristic function of the interval $[-a,a]$. Consider the Banach space
\begin{equation}
  \label{eq:Bs}
  \mathcal{B}=\{u\in C(\RR)|\|u\|_{\nu}:=\sup_{t\in\RR^+}e^{-\nu t}\|u(t,\cdot)\|_1<\infty\};\ \ (\nu>\sqrt 2\|V\|_{\infty}^{\frac12})
\end{equation}
Applying Fubini to integrate first in $x$, we see that
$\|\int_{-\infty}^\infty \chi_t(y-x)g(y)dy\|_1\leqslant  2t\|g\|_1$ and (since by
definition
$\|u(\cdot,s)\|_1\leqslant  \|u\|_{\nu}e^{\nu s}$)
\begin{multline}
  \label{eq:eqFb1}
 \sup_{t>0}e^{-\nu t} \left\|\int_0^t\int_{-\infty}^\infty u(y,s)V(y)\chi_{t-s}(y-x)dyds
  \right\|_1 \\\leqslant y \|V\|_\infty \|u\|_{\nu}\sup_{t>0}e^{-\nu t}\int_0^t 2(t-s)e^{\nu s}ds\le
  2\|V\|_\infty\nu^{-2} \|u\|_{\nu}
\end{multline}
Using \eqref{eq:eqFb1} we see that    $\mathcal{A}:\mathcal{B}\to \mathcal{B}$
is contractive. Also, assuming $f, g$  and $V$ are
smooth, the solution is seen to be  smooth too:
 since $u\in L^1$,  Duhamel's formula shows that it is
continuous;  then, as usual,  using
continuity we  derive differentiability, and  inductively, we see that $u$ is smooth.
\end{proof}

\begin{proof}[Proof of Lemma \ref{nm0}]
This is by straightforward calculation. For  $n\geqslant 0$ we get
\begin{multline}
\label{inttau0}
\int_3^{\infty}e^{-\epsilon \tau x} \tau^n(\ln \tau)^ld\tau=\left(\int_0^{\infty}-\int_0^{3}\right)e^{-\epsilon \tau x} \tau^n(\ln \tau)^ld\tau=\frac{1}{(\epsilon x)^{n+1}}\int_0^{\infty}e^{-u}u^n(\ln u\\-\ln(\epsilon x))^ldu-\int_0^{3}e^{-\epsilon \tau x} \tau^n(\ln \tau)^ld\tau
=\frac{1}{(\epsilon x)^{n+1}}\sum_{q=0}^{l}c_{q}^{(n;l)}(\ln (\epsilon x))^q-\int_0^{3}e^{-\epsilon \tau x} \tau^n(\ln \tau)^ld\tau
\end{multline}
where the last term is $R_{a,n}$ and \eqref{unibd1} is immediate, and
$$ c_{q}^{(n;l)}=\left(\int_0^{\infty}e^{-u}u^n(\ln u)^{l-q}du\right)\frac{(-1)^{q} l!}{(l-q)!q!}$$
In particular $c_{1}^{(0;1)}=-1$.

Now \eqref{nm1} for $n=-1$ follows from integration by parts
\begin{equation}\label{nm20}
\int_3^{\infty}e^{-\epsilon \tau x}\tau^{-1}(\ln \tau)^l d\tau=-\frac{1}{l+1}e^{-3\epsilon x}(\ln 3)^{l+1}+\frac{\epsilon x}{l+1}\int_3^{\infty}e^{-\epsilon \tau x}(\ln \tau)^{l+1}d\tau
\end{equation}
where the first term satisfies \eqref{unibd1}, and the last integral in \eqref{nm20} was evaluated in \eqref{inttau0}.

For $n<-1$ we have by integration by parts
\begin{multline}\label{nm10}
\int_3^{\infty}e^{-\epsilon \tau x}\tau^{n}(\ln \tau)^l d\tau=-\frac{1}{n+1}e^{-3\epsilon x}3^{n+1}(\ln 3)^{l}\\-\frac{l}{n+1}\int_3^{\infty}e^{-\epsilon \tau x}\tau^{n}(\ln \tau)^{l-1}d\tau
+\frac{\epsilon x}{n+1}\int_3^{\infty}e^{-\epsilon \tau x}\tau^{n+1}(\ln \tau)^{l}d\tau
\end{multline}
and \eqref{nm1} follows by induction on $l$ and $n$ using \eqref{nm10} and integration by parts.
\end{proof}

\section{Acknowledgments.}
This work was supported in part by the National Science Foundation
 DMS-01108794  (OC).

\end{document}